\documentclass[11pt]{amsart}

\usepackage[T1]{fontenc}
\usepackage[utf8]{inputenc}
\usepackage{lmodern}
\usepackage{layout}
\usepackage[american]{babel}
\usepackage[babel, final]{microtype}
\usepackage{amssymb}
\usepackage{amscd}
\usepackage[mathscr]{eucal}
\usepackage{amsmath}
\usepackage{outlines}
\usepackage{enumitem}

\usepackage{float}

\usepackage[pdftex, dvipsnames]{xcolor}
\usepackage[pdftex, final]{graphicx}
\usepackage{caption}
\usepackage{subcaption}
\usepackage{pinlabel}
\usepackage[pdftex, letterpaper, includehead, includefoot, nomarginpar, 
margin=1.1in]{geometry}
\definecolor{mypurple}{RGB}{255,60,0}
\definecolor{myorange}{RGB}{40,113,176}
\usepackage[pdftex, final, colorlinks=true, urlcolor=myorange, 
linkcolor=myorange, citecolor=mypurple, filecolor=mypurple, menucolor=mypurple, 
bookmarks=true, bookmarksdepth=3, bookmarksnumbered=true, bookmarksopen=true, 
bookmarksopenlevel=2]{hyperref}
\hypersetup{
  pdftitle={Surgery on conformally Anosov flows},
  pdfauthor={Federico Salmoiraghi},
  pdfsubject={Surgery, Legendrian tranverse Knots, Anosov Flows, Contact 3-Manifolds},
  pdfkeywords={ Surgery, Legendrian tranverse Knots, Anosov Flows, Contact 3-Manifolds
    57M27, 57R58}
}
\usepackage{tikz}
\usetikzlibrary{arrows, automata}
\usetikzlibrary{cd}

\usepackage[final]{showlabels} %Use "inline" to show labels, "final" to hide

\makeatletter
\g@addto@macro\@floatboxreset\centering
\makeatother

\allowdisplaybreaks[4]

\addto\extrasamerican{%
}
\let\fullref\autoref
\newtheorem{maintheorem}{Theorem}
\newtheorem{maincorollary}{Corollary}

\newtheorem{conjecture}{Conjecture}[section]

\newtheorem{theorem}{Theorem}[section]
\newtheorem{corollary}{Corollary}[section]
\newtheorem{proposition}{Proposition}[section]
\newtheorem{lemma}{Lemma}[section]

\theoremstyle{definition}
\newtheorem{definition}{Definition}[section]
\newtheorem{remark}{Remark}[section]
\newtheorem{example}{Example}[section]

\makeatletter
\let\c@maincorollary=\c@maintheorem
\let\c@corollary=\c@theorem
\let\c@proposition=\c@theorem
\let\c@lemma=\c@theorem
\let\c@remark=\c@theorem
\let\c@definition=\c@theorem
\let\c@example=\c@theorem
\let\c@question=\c@theorem
\makeatother
\def\makeautorefname#1#2{\expandafter\def\csname#1autorefname\endcsname{#2}}
\makeautorefname{maintheorem}{Theorem}%
\makeautorefname{maincorollary}{Corollary}%
\makeautorefname{theorem}{Theorem}%
\makeautorefname{corollary}{Corollary}%
\makeautorefname{proposition}{Proposition}%
\makeautorefname{lemma}{Lemma}%
\makeautorefname{remark}{Remark}%
\makeautorefname{definition}{Definition}%
\makeautorefname{example}{Example}%
\makeautorefname{equation}{equation}%

\DeclareMathOperator{\PSL}{\mathrm{PSL}}

\newcommand{\A}{\mathcal{A}}

\begin{document}
%%%%%%%%%%%%%%%%%%%%%%%%%%%%%%%%%%%%%%%%%%%%%%%%%%%%%%%

%\thispagestyle{empty}
%
\title[Surgery on Anosov flows using bi-contact geometry]{Surgery on Anosov flows using bi-contact geometry}

\author{Federico Salmoiraghi}\thanks{This research was supported by the Israel Science Foundation (grant No. 51/4051)}
\address{Department of Mathematics \\ Queen's University \\Kingston, Canada}
\email{\href{mailto:federico.salmoiraghi@queensu.ca}{federico.salmoiraghi@queensu.ca}}

%\keywords{Heegaard Floer homology}
%\subjclass[2010]{57M27; 57R58}

%%%%%%%%%%%%%%%%%%%%%%%%%%%%%%%%%%%%%%%%%%%%%%%%%%%%%%%

%%%%%%%%%%%%%%%%%%%%%%%%%%%%%%%%%%%%%%%%%%%%%%%%%%%%%%%
\begin{abstract}
Using bi-contact geometry, we define a new type of Dehn surgery on an Anosov flow with orientable weak invariant foliations. The Anosovity of the new flow is strictly connected to contact geometry and the Reeb dynamics of the defining bi-contact structure. %In particular we show that near a closed orbit the new construction produces hyperbolicity regardless of the sign of the Dehn twist.
This approach gives new insights into the properties of the flows produced by Goodman surgery and clarifies under which conditions Goodman's construction yields an Anosov flow. Our main application gives a necessary and sufficient condition to generate a contact Anosov flow by Foulon-Hasselblatt Legendrian surgery on a geodesic flow. In particular we show that this is possible if and only if the surgery is performed along a simple closed geodesic. As a corollary we have that any (positive) skewed $\mathbb{R}$-covered Anosov flow obtained by surgery on a closed orbit of a geodesic flow is orbit equivalent to a (positive) contact Anosov flow.  %on a geodesicit can be interpreted as the counterpart in the contact category of results proven recently by  Bonatti-Iakovoglou \cite{BoIo} and Marty \cite{Ma} in the category of $\mathbb{R}$-covered Anosov flows. As a corollary 
\end{abstract}
%%%%%%%%%%%%%%%%%%%%%%%%%%%%%%%%%%%%%%%%%%%%%%%%%%%%%%%

\maketitle

%%%%%%%%%%%%%%%%%%%%%%%%%%%%%%%%%%%%%%%%%%%%%%%%%%%%%%%

%\tableofcontents

%%%%%%%%%%%%%%%%%%%%%%%%%%%%%%%%%%%%%%%%%%%%%%%%%%%%%%%
\section{Introduction} % (fold)
\label{sec:intro}
%%%%%%%%%%%%%%%%%%%%%%%%%%%%%%%%%%%%%%%%%%%%%%%%%%%%%%%

The use of surgery has tremendously advanced our understanding of Anosov flows on $3$-manifolds. In the groundbreaking work of Handel and Thurston \cite{HaTh} the first example of non-algebraic transitive Anosov flow was constructed performing surgery on a geodesic flow. Inspired by Handel and Thurston work, Goodman \cite{Goo} defined a Dehn type surgery near a closed orbit and constructed the first examples of an Anosov flow on a hyperbolic 3-manifold.

In the present work we introduce a new type of Dehn surgery on Anosov flows using a supporting {\it bi-contact structure}.
Mitsumatsu \cite{Mit} first noticed that the generating vector field of an Anosov flow belongs to the intersection of a pair of transverse contact structures $(\xi_-,\xi_+)$ rotating towards each other along the flow and asymptotic to the stable and unstable foliations. These pairs are called bi-contact structures and the associated flows are called {\it projectively Anosov flows}. Projectively Anosov flows are more abundant than Anosov flows: it can  be shown indeed that every closed 3-manifolds admits a pA flow while not every 3-manifold admits an Anosov flow.

A knot $K$ in a bi-contact structure is called {\it Legendrian-transverse} if its tangents are contained in one of the contact structures and they are transverse to the other one. There is plenty of Legendrian-transverse knots in a bi-contact structure supporting an Anosov flows. Recently Hozoori \cite{Hoz2} has shown that if a flow is Anosov with orientable weak invariant foliations there are supporting bi-contact structures such that near any periodic orbit $\gamma$ there is a Legendrian-transverse knot isotopic to $\gamma$. Additionally there are many Legendrian-tranverse knots that are not isotopic to any closed orbit. Examples of these knots can be found on the torus fiber of a Anosov suspension flow or considering the circles of the Seifert fibration on the unit tangent bundle of an hyperbolic surface.  From now on we assume that $K$ is a knot that is Legendrian for $\xi_-$ and transverse to $\xi_+$. 
\begin{maintheorem}
\label{thm:1}
Let $K$ be a Legendrian-transverse knot in a bi-contact structure $(\xi_-,\xi_+)$ defining a volume preserving Anosov flow. There is a $(1,q)$-Dehn surgery along an annulus $A_0$ tangent to the flow that yields new bi-contact structures for every $q\in \mathbb{N}$.
\end{maintheorem}
\fullref{thm:1} can be extended to more general bi-contact structures. For example it also holds near a closed orbit of a general Anosov flow with orientable weak invariant foliations.

Throughout the rest of the paper we will refer to the construction of \fullref{thm:1} as {\it bi-contact surgery}.
For $q<0$ integer, the smooth plane field  distribution produced by the construction of \fullref{thm:1} is not in general a bi-contact structure.

Since being defined by a bi-contact structure is not a sufficient condition to be Anosov, it is not immediately clear if the bi-contact structures produced in \fullref{thm:1} support an Anosov flow. 
The following result shows that the Anosovity of the new flow is strictly connected to bi-contact geometry.

\begin{maintheorem}
\label{thm:2}

Let $K$ be a Legendrian-transverse knot in a bi-contact structure defining a volume preserving Anosov flow. If a bi-contact surgery yields a bi-contact structure it yields an Anosov flow. 
\end{maintheorem}

In particular, a $(1,q)$-bi-contact surgery on a bi-contact structure defining a volume preserving Anosov flow yields new Anosov flows for every $q\in \mathbb{N}$.
 
The construction of \fullref{thm:1} naturally fits into the framework set up by Hozoori in \cite{Hoz} where is given a symplectic characterization of the bi-contact structures that define Anosov flows. Using this characterization we give a proof of  \fullref{thm:2} that relies only on the properties of the Reeb dynamics of an underlying bi-contact structure \footnote{In particular we do not use the {\it cone field criterion of hyperbolicity.}}. 
\\

The construction of \fullref{thm:1} 
gives new insight into the properties of the flows produced by Goodman surgery in virtue of the following equivalence.

\begin{maintheorem}
\label{thm:2.1}
There are choices such that a flow generated by a bi-contact $(1,q)$-surgery on a volume preserving Anosov flow $\phi^t$ is orbit equivalent to a flow constructed by $(1,q)$-Goodman surgery on $\phi^t$.

\end{maintheorem}

Goodman's operation consists in cutting an Anosov flow along a transverse annulus $C$ in a neighborhood of a closed orbit and gluing back the two sides of the cut adding a $(1,q)$-Dehn twist. Goodman proved that such an annulus comes with a {\it preferred direction} in the sense that there is a sign of the twist that always produces hyperbolicity. For example, along an annulus with positive preferred direction a $(1,q)$-Goodman surgery produces an Anosov flow for every integer $q>0$. 
   
For a fixed transverse annulus with positive preferred direction it is not known in general if a $(1,q)$-Goodman surgery with $q<0$ integer generates hyperbolicity. \fullref{thm:2} and \fullref{thm:2.1} allows us to study the Anosovity of the flows generated by surgeries performed in the direction opposite to the preferred one using tools of contact geometry.
\\

Let $C$ be an embedded transverse annulus with positive preferred direction centered on a Legendrian-transverse knot $K$ (non necessarily isotopic to a closed orbit). To a neighborhood $N$ of $C$ we associate a positive real number $k$ depending just on the behaviour of the contact structure $\xi_+$ on $\partial N$. This number can be interpreted as the {\it slope} of the {\it characteristic foliation} induced by $\xi_+$ on $\partial N$. For a volume preserving Anosov flow we have the following.

\begin{maintheorem}
\label{thm:3}
Let $k>0$ be the slope of the characteristic foliation induced by $\xi_+$ on $\partial N$. For every $q>-k$ a $(1,q)$-Goodman surgery along $C$ produces an Anosov flow.
\end{maintheorem}

Let $C_q$ be a transverse annulus with positive preferred direction such that a $(1,q)$-Goodman surgery for a fixed $q<0$ produces an Anosov flow. Foulon, Hasselblatt and Vaugon show in \cite{FHV2} that for a sufficiently thin annulus $C\subset C_{q}$ centered on $K$, a $(1,q)$-Goodman surgery along $C$ produces a flow that is not Anosov. We show that the situation radically changes after a local perturbation of the flow. 

\begin{maincorollary} 
\label{cor:4}
Suppose that $C_{q}$ is an annulus with positive preferred direction transverse to a volume preserving Anosov flow $\phi^t$. Suppose also that a $(1,q)$-Goodman surgery along $C_q$ produces an Anosov flow for a fixed $q<0$.
\begin{enumerate}
\item There is an arbitrarily thin annulus $C\subset C_{q}$ and a flow $\phi_t'$ orbit equivalent to $\phi_t$ such that a $(1,q)$-Goodman surgery along $C$ generates an Anosov flow. 
\item There is a $C^1$-path of Anosov flows connecting $\phi_t$ and $\phi_t'$.
\item The flow-lines of $\phi'_t$ coincide with the ones of $\phi_t$ outside an arbitrarily small flow-box neighborhood $N$ of $C_{q}$.
\end{enumerate}
\end{maincorollary}

The proof of \fullref{cor:4} uses the interaction between the structural stability of an Anosov flow and the flexibility of its underlying bi-contact structure to construct $C^1$ paths of Anosov flows.
\\

Foulon, Hasselblatt and Vaugon furthermore show in \cite{FHV2} that in any Anosov flow for a fixed embedded transverse annulus $C$ with positive preferred direction there is a large enough $q<0$ such that a $(1,q)$-Goodman surgery produces a flow that is not Anosov. Let $\phi^t$ be a volume preserving Anosov flow with weak orientable invariant foliations. If there is an embedded {\it quasi-transverse} annulus $ C_{-\infty}$ bounded on one side by a Legendrian-transverse knot $K$ and on the other side by a closed orbit $\gamma$ we have the following. 

\begin{maincorollary}
\label{cor:5}
There is a nested sequence $\{C_r\}_{r\in \mathbb{Z}^-}$ of transverse annuli with positive preferred direction bounded by $K$ on one side and such that a $(1,r)$-Goodman surgery along $C_r$ yields an Anosov flow. Moreover we have $$ C_{-\infty}=\overline{\bigcup_{r\in \mathbb{Z}^-}C_r}.$$ 
\end{maincorollary}

In particular, given an annulus $C_r$ of the sequence, we obtain the annulus $C_{r-1}$ extending $C_r$ towards the closed orbit $\gamma$.  
\\

We use \fullref{thm:3} and its corollaries to study surgeries in the context of {\it contact Anosov flows}, that are Anosov flows preserving a contact form. The condition of being contact Anosov has a number of remarkable geometric consequences (see \cite{Ham}, \cite{Liv} and \cite{Fan}) including exponential decay of correlations. In the context of $3$-manifold, Barbot has shown that every contact Anosov flow is {\it skewed} $\mathbb{R}$-covered \footnote{An Anosov flow is {\it $\mathbb{R}$-covered} if the stable (or the unstable) weak foliation lifts to a foliation in the universal cover which has leaf space homeomorphic to $\mathbb{R}$.}. The relation between these two classes of flows is expected to be stronger.

\begin{conjecture}[Barbot--Barthelmé]
\label{conj1}
If $\phi^t$ is a (positively) skewed $\mathbb{R}$-covered Anosov flow, then it is orbit equivalent to a contact Anosov flows.\footnote{Marty \cite{Marty2} posted a proof of this statement after the completion of the present work. }
\end{conjecture}

For almost half of a century since the seminal work of Anosov, the only known examples of contact Anosov flows were the geodesic flows of Riemannian or Finsler manifolds. 
Foulon and Hasselblatt  \cite{FoHa1} showed that it is not only possible to construct new examples of contact Anosov flows performing Goodman surgery on a geodesic flow, but also that it could be done producing hyperbolic manifolds.

The construction of Foulon and Hasselblatt uses a special class of arbitrarily thin transverse annuli with positive preferred direction containing a knot $L$ that is Legendrian for the contact structure preserved by the geodesic flow. Such annuli are located far from a closed orbit therefore the set of negative surgery coefficients that yield a contact Anosov flow is not known. 
If the knot $L$ is associated to a simple closed geodesic there is an embedded quasi-transverse annulus containing $L$ and \fullref{cor:5} applies. As a consequence we have the following.

\begin{maintheorem}
\label{thm:6}
Chose a closed orbit $\gamma$ in the geodesic flow on the unit tangent bundle of a hyperbolic surface $S$. If $q<0$ a $(1,q)$-Goodman surgery generates a flow that is orbit equivalent to a contact Anosov \footnote{Our proof of Anosovity does not rely on Barbot's version of the cone field criterion \cite{Bar1} and can be used to give an alternate proof of Theorem 4.3 in \cite{FoHa1}.} flow if and only if the orbit $\gamma$ is a lift of a simple closed geodesic on $S$.
\end{maintheorem}

\fullref{thm:6} represents the counterpart in the category of contact Anosov flows of results recently proven by  Bonatti-Iakovoglou \cite{BoIo} (Therem 1) and Marty \cite{Ma} (Theorem J) in the category of $\mathbb{R}$-covered Anosov flows. Using \fullref{thm:6} and work of Asaoka--Bonatti--Marty \cite{ABM} we prove a version of Conjecture \fullref{conj1} for surgeries along a single closed orbit of the geodesic flow. Let $\phi^t$ be the geodesic flow with the orientation that makes it a positively skewed $\mathbb{R}$-covered Anosov flow.

\begin{maintheorem}
Any (positive) skewed $\mathbb{R}$-covered Anosov flow obtained by surgery along a closed orbit of a geodesic flow is orbit equivalent to a (positive) contact Anosov flow.
\end{maintheorem}

Note that \fullref{thm:6} has a natural interpretation in the context of contact and symplectic geometry. Foulon and Hasselblatt construction is an example of a {\it contact} surgery in the sense that given a Legendrian knot $L$ in a contact 3-manifold $(M,\eta)$ it produces a new contact 3-manifold performing a Dehn type surgery along $L$. Historically a $(1,-1)$-contact surgery is called {\it Legendrian} surgery. Legendrian surgeries preserve tightness \cite{Wan1} and have a natural interpretation from a symplectic point of view. Foulon and Hasselblatt show \cite{FoHa1} that it is possible to generate an hyperbolic manifold with a (positive) contact Anosov flow performing a single positive contact surgery on a closed {\it filling} geodesic. Since a filling geodesic is non simple, a consequence of \fullref{thm:6} is the impossibility of generating an hyperbolic manifold with a (positive) contact Anosov flow performing a single Legendrian surgery on a geodesic flow. On the contrary it is possible to generate a (positive) contact Anosov flow performing Legendrian surgery along any simple closed geodesic. This operation yields a graph manifold and the associated flows are orbit equivalent to the ones described by Handel and Thurston in \cite{HaTh} for a shear with $a_j<0$ and short enough geodesic $\gamma_j$.

\subsection{Acknowledgment}

The author would like to express his gratitude to Tali Pinsky for introducing him to this subject and for her support throughout the development of this project. He also would like to thank Thomas Barthelmé, Surena Hozoori, Anne Vaugon, Théo Marty and Martin Mion-Mouton for the illuminating conversations and their patience in explaining to the author various subtlety of the subject.

\section{Anosov flows, projectively Anosov flows and bi-contact structures} 
\label{sec:Anosov}

Anosov flows are an important class of dynamical system characterized by structural stability under $C^1$-small perturbations (see \cite{Ano1}, \cite{Ano2} and \cite{Pla}). Beyond their interesting dynamical properties there are evidences of an intricate and beautiful relationship with the topology of the manifold they inhabit (see the survey \cite{Ber} for classic and more recent developments \cite{BaFe1}, \cite{BaFe2}, \cite{BaFe3}, \cite{BaFe4} by Fenley, Barbot and Barthelmé). Geometrically they are distinguished by the contracting and expanding behaviour of two invariant directions 

\begin{definition}
Let $M$ be a closed manifold and $\phi^t:M \rightarrow M$ a $C^1$ flow on $M$. The flow $\phi^t$ is called {\it Anosov} if there is a splitting of the tangent biundle $TM=E^{uu} \oplus E^{ss} \oplus \langle X \rangle$ preserved by $D\phi^t$ and positive constants $A$ and $B$ such that
$$\lVert D \phi^t(v^u) \rVert\geq Ae^{Bt} \lVert  v^u\rVert\;\;\;\;\;\;\text{for any}\;v^u \in E^{uu}$$
$$\lVert D \phi^t(v^s) \rVert\leq Ae^{-Bt} \lVert  v^s\rVert\;\;\;\;\text{for any}\;v^s \in E^{ss}.$$
Here $\lVert \cdot \rVert$ is induce by a Riemmanian metric on $TM$.
We call $E^{uu}$ and $E^{ss}$ respectively the {\it strong unstable bundle} and the {\it strong stable bundle}. 

\end{definition}

Classic examples of Anosov flows are the geodesic flow on the unit tangent bundle of a hyperbolic surface and the suspension flows of hyperbolic linear automorphisms of the torus.

The definition above has remarkable geometric consequences. Anosov showed that the distributions $E^{ss}$ and $E^{uu}$ are uniquely integrable and the associated foliations are denoted by $\mathcal{F}^{ss}$ and $\mathcal{F}^{uu}$. Moreover, the {\it weak stable bundle} $E^s=E^{ss} \oplus \langle X \rangle$ and the {\it weak unstable bundle} $E^u=E^{uu} \oplus \langle X \rangle$ are also uniquely integrable and the codimension one associated foliations are dented with $\mathcal{F}^{s}$ and $\mathcal{F}
^{u}$ (see \fullref{bundles1}).

\begin{figure}
\label{bundles1}
\labellist 
     \begin{footnotesize}
     \pinlabel $\mathcal{F}^{u}$ at 36 188
     \pinlabel $\mathcal{F}^{s}$ at 280 10
     \pinlabel $X$ at 3 85
     \end{footnotesize}
\endlabellist
\includegraphics[width=0.9\textwidth]{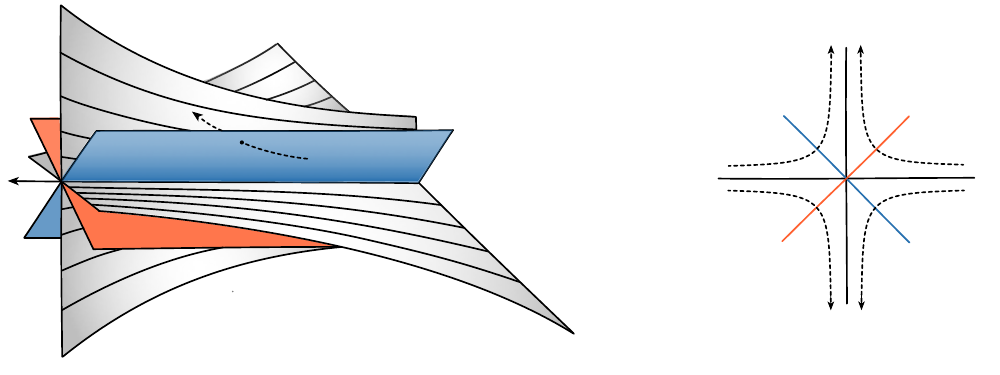}
  \caption{On the left, the tangent bundle in a neighborhood of a flow line of an Anosov flow. In red and blue are depicted the plane fields defining the bi-contact structure. The vertical plane is the the leaf of the unstable weak foliation $\mathcal{F}^{u}$. On the right, the normal bundle $TM/\langle X \rangle$.}
\label{bundles1}
\end{figure}

Mitsumatsu \cite{Mit} first noticed that an Anosov flow with orientable weak invariant foliations is tangent to the intersection of two transverse contact structures (see also Eliashberg-Thurston \cite{ElTh}). We will call such pairs {\it bi-contact structures}. However, the converse statement is not true and there are bi-contact structures that do not define Anosov flows. 
\subsection{Contact structures and Reeb flows.}
A co-oriented contact structure is a plane field distribution that is maximally non-integrable in the sense that it can be described as the kernel of a $C^1$ 1-form satisfying the relation $\alpha\wedge d\alpha\neq0$. By Frobenious theorem contact structures can be thought as polar opposite of foliations: there is not a subsurface $S$ such $TS=\ker \alpha$, even in a neighborhood of a point. Contact structures come in two types, positive and negative. A positive contact structure is a plane field distribution $\xi_+$ described by a $C^1$ 1-form satisfying the relation $\alpha_+\wedge d\alpha_+>0$. A negative contact structure $\xi_-$ is described instead by a $C^1$ 1-form such that $\alpha_-\wedge d\alpha_-<0$. 

An important property of contact structures is that they do not have local invariants. Indeed by Darboux theorem positive (negative) contact structures are all locally contactomorphic to the positive (negative) {\it standard contact structure in} $\mathbb{R} ^3$ described by $\xi_{std}^+=\ker dz-y\:dx$ and $\xi_{std}^-=\ker dz+y\:dx$. Therefore we can locally picture a positive contact (negative) structure as a plane field whose plane rotate counterclockwise (clockwise) along the $x$-axis. Associated to the defining contact forms there is an important class of flows called {\it Reeb flows}. Given a contact form $\alpha$ we define the Reeb vector field $R_{\alpha}$ as the unique vector field satisfying the equations

$$\alpha(R_{\alpha})=1,\;\;\;\;d\alpha(R_{\alpha},\cdot)=0.$$
These relations imply that $R_{\alpha}$ is transverse to $\ker \alpha$ and $\mathcal{L}_{R_{\alpha}}\alpha=0$ ($R_{\alpha}$ preserves $\alpha$). 

We conclude this section introducing a very important class of vector fields that is well studied and has remarkable geometric properties (\cite{Liv}, \cite{FoHa1}).

\begin{definition}
A Reeb flow that is also an Anosov flow is called a {\it contact Anosov} flow. 
\end{definition}

In other words a contact Anosov flow is an Anosov flow preserving a contact form.

\subsection{Bi-contact structures and projectively Anosov flows.}
We now give an example of a pair of opposite and transverse contact structures that does not define an Anosov flow. 
\begin{example}
\label{ex:T3}
We construct a bi-contact structure on $T^3$ using a recipe introduced by Mitsumatsu in \cite{Mit} and \cite{Mit2}.
Consider the contact forms defined on $T^2\times I$
$$\alpha_{n}=cos(2n\pi z)dx-sin(2n\pi z)dy,$$ $$\alpha_{-m}=cos(2m\pi z)dx+sin(2m\pi z)dy.$$
They are not transverse to each other on the tori defined by  $\{z=0\}$, $\{z=\frac{1}{4}\}$, $\{z=\frac{1}{2}\}$, $\{z=\frac{3}{4}\}$. If we introduce a perturbation $\epsilon(z)\: dz$ such that $\epsilon(z)$ is a function that doesn't vanish on the tori, the contact forms $\alpha_+=\alpha_{n}+\epsilon(z)\: dz$ and $\alpha_-=\alpha_{-m}$ define transverse contact structures of opposite orientations.

As Plante and Thurston showed in \cite{PlTh}, the fundamental group of a manifold that admits an Anosov flow grows exponentially. Since the ambient manifold is $T^3$, any flow defined by the pair of contact structures $(\ker \alpha_-,\ker \alpha_+)$ does not define an Anosov flow.

\end{example}

\begin{definition}[Mitsumatsu \cite{Mit}]
Let $M$ be a closed manifold and $\phi^t:M \rightarrow M$ a $C^1$ flow on $M$. The flow $\phi^t$ is called {\it projectively Anosov} if there is a splitting of the projectified tangent bundle $TM/ \langle X \rangle=\mathcal{E}^u \oplus \mathcal{E}^s$ preserved by $D\phi^t$ and positive constants $A$ and $B$ such that
$$\frac{\lVert D \phi^t(v^u) \rVert}{ \lVert D \phi^t(v^s) \rVert}\geq Ae^{Bt} \frac{\lVert  v^u\rVert}{\lVert  v^s\rVert}\;\;\;\;\;\;\;\text{for any}\;v^u \in \mathcal{E}^u \; \text{and}\;v^s \in \mathcal{E}^s $$

Here $\lVert \cdot \rVert$ is induce by a Riemmanian metric on $TM$.
We call $\mathcal{E}^u$ and $\mathcal{E}^s$ respectively the {\it unstable bundle} and the {\it stable bundle}. 

\end{definition}

The invariant bundles $\mathcal{E}^u$ and $\mathcal{E}^s$ induce invariant plane fields $E^u$ and $E^s$ on $M$. These plane fields are continuous and integrable, but unlike the Anosov case, the integral manifolds may not be unique (see \cite{ElTh}).
However, when they are smooth they also are uniquely integrable. We call these flows {\it regular projectively Anosov} (see \cite{Nod}, \cite{NoTs} and \cite{Asa} for a complete classification).
The following result gives a geometric characterization of the family of projectively Anosov flows 
\begin{proposition}[Mitsumatsu \cite{Mit}]
Let $X$ be a $C^1$ vector field on $M$. Then $X$ is projectively Anosov if and only if it is defined by a bi-contact structure.
\end{proposition}

\begin{remark}
 We refer to \cite{Hoz} and \cite{Hoz2} for a more complete overview of the connection between bi-contact geometry, symplectic geometry and projectively Anosov flows and for a precise discussion on the regularity of the plane fields, bundles and foliations involved. 
\end{remark}
\subsection{Reeb dynamics of a bi-contact structure defining an Anosov flows}

We now present a characterization of Anosov flows due to Hoozori \cite{Hoz} that uses the Reeb dynamics of the underling bi-contact structure.

\begin{definition}[Hozoori \cite{Hoz}]
Consider a bi-contact structure $(\xi_-,\xi_+)$. A vector field is {\it dynamically positive (negative)} if at every point $p\in M$ it lies in the interior of the first or third (second or forth) region in \fullref{bundles1}.
\end{definition}

\begin{theorem}[Hozoori \cite{Hoz}]
\label{prop:Hoz}
Let $\phi^t$ be a projectively Anosov flow on $M$. The following are equivalent.
\end{theorem}

\begin{enumerate}
\item {The flow \it $\phi^t$ is Anosov},

\label{cond:h1}

\item {\it There is a pair} $(\alpha_-,\alpha_+)$ {\it of positive and negative contact forms defining} $\phi^t$ {\it such that the Reeb vector field of} $\alpha_+$ {\it is dynamically negative},

\label{cond:h2}

\item {\it There is a pair} $(\alpha_-,\alpha_+)$ {\it of positive and negative contact forms defining} $\phi^t$ {\it such that the Reeb vector field of} $\alpha_-$ {\it is dynamically positive}.

\label{cond:h3}
\end{enumerate}

A {\it volume preserving} Anosov flow is an Anosov flow preserving a continuous volume form. It is known that if the flow is $C^k$ such a volume form is automatically $C^k$ (see \cite{Hoz2} for more references).

 Hozoori shows \cite{Hoz2} the following characterization of bi-contact structures defining volume preserving Anosov flows.

\begin{theorem}
[Hozoori \cite{Hoz2}]
\label{VPflows}
Let $\phi^t$ be a projectively Anosov flow on $M$. $\phi^t$ is a volume preserving Anosov flow if and only if there are pairs of positive and negative contact forms defining $\phi^t$ such that $R_{\alpha_-}\in \ker \alpha_+$ and $R_{\alpha_+}\in \ker \alpha_-$.

\end{theorem}

\begin{figure}

\labellist 
     \begin{footnotesize}
     \pinlabel $L$ at 260 152
     \pinlabel UT$\gamma$ at 448 155
     \pinlabel $\gamma^-$ at 268 38
    \pinlabel $\gamma^+$ at 268 275
    \pinlabel $X$ at 185 174
     \end{footnotesize}
\endlabellist

\includegraphics[width=1\textwidth]{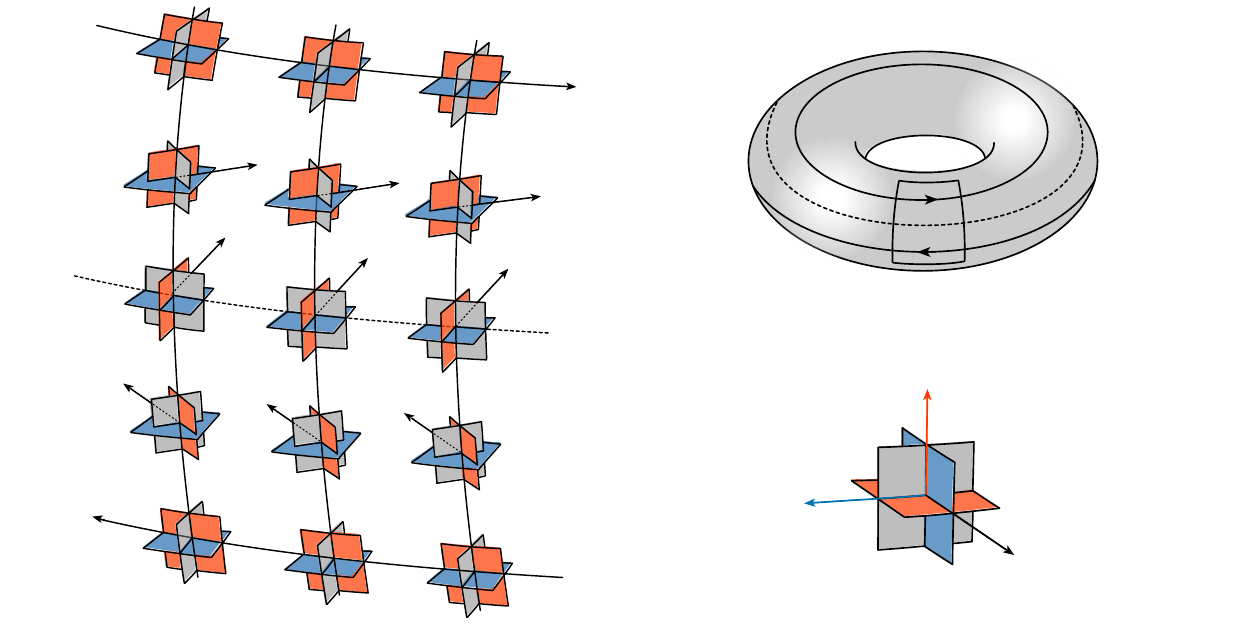}
 \caption{On the top right, a Birkhoff torus associated to simple closed geodesic $\gamma$. On the right side, the contact structures described in \ref{Bir}. In grey, the contact structure $\eta_+=\ker\beta_+$ preserved by the geodesic flow. The closed orbits $\gamma^+$ and $\mathfrak{\gamma^-}$ are Legendrian knots for the bi-contact structure $(\xi_-,\xi_+)$ that generates the geodesic flow. The contact structure $\xi_+=\ker\alpha_+$ is represented in red while the $\xi_-=\ker\alpha_-$, is represented in blue. The special knot $L$ is represented by the dashed line.}
\label{fig:Geodesic flow}
\end{figure}

\section{The geodesic flow on the unit tangent bundle of a hyperbolic surface} 
\label{sec:Geodesic}

In this section we introduce the prototype of an Anosov flow: the geodesic flow on the unit tangent bundle of a hyperbolic surface (see \cite{Ano1}). This object is particularly important in our context since it is the ambient flow where several surgery constructions (i.e. the ones introduced by Handel-Thurston and Foulon-Hasselblatt) are defined. Moreover for a long time it was the only known example of a contact Anosov flow.

\begin{subsection}{Geometric structures on UTS}
\label{Additional}

The geodesic flow on the unit tangent bundle of a hyperbolic surface $S$ carries some remarkable geometric structure that can be interpreted in the context of bi-contact geometry. We follow the discussion of \cite{FHV2}. Using the identification of UT$\mathbb{H}^2$ with $\PSL(2,\mathbb{R})$ it is possible to show that there is a canonical framing  consisting of the vector field $X$ that generates the flow, the periodic vector field $V$ pointing in the fiber direction, and the vector field defined by $H:=[V,X]$. This frame satisfies the following relations 
\begin{equation}
\label{str eq}
[V,X]=H,\;\;\;\;[H,X]=V,\;\;\;\;[H,V]=X.
\end{equation}
A consequence of the structure equations is that the strong stable and unstable bundles $E^{\pm}$ are spanned by the vectors $e^{\pm}=V\pm H$.
It is not difficult to show that (\ref{str eq}) implies the existence of three 1-forms $\alpha_-,\alpha_+$ and $\beta_+$ defining mutually transverse contact structures (see \fullref{fig:Geodesic flow} and \cite{FHV2} for more details). They are defined by the following

$$\beta_+(V)=0=\beta_+(H),\;\;\;\;\alpha_+(X)=0=\alpha_+(V),\;\;\;\;\alpha_-(X)=0=\alpha_-(H),$$
 $$\beta_+(X)=1,\;\;\;\;\;\;\;\;\;\;\;\;\;\;\;\;\;\;\;\alpha_+(H)=1,\;\;\;\;\;\;\;\;\;\;\;\;\;\;\;\;\;\;\;\alpha_-(V)=1,$$
 $$d\beta_+(X,\cdot)=0,\;\;\;\;\;\;\;\;\;\;\;\;\;\;\;d\alpha_+(H,\cdot)=0,\;\;\;\;\;\;\;\;\;\;\;\;\;\;\;d\alpha_-(V,\cdot)=0.$$

The relations above show that the vector fields $X$, $H$, and $V$ are respectively the Reeb vector fields of the contact forms $\beta_+$, $\alpha_+$ and $\alpha_-$. In particular the geodesic flow is the Reeb flow of $\beta_+$. The pair of contact forms $(\alpha_-,\alpha_+)$ define a bi-contact structure that supports $X$.

\end{subsection}

\begin{figure}

\labellist 
     \begin{footnotesize}
     \pinlabel $L$ at 208 135
      \pinlabel $\gamma^+$ at 210 12
\pinlabel $\gamma^+$ at 210 175      
       \pinlabel $\gamma^-$ at 210 95

     \pinlabel $\gamma^+$ at 465 12
     \pinlabel $\gamma^+$ at 465 175
    \pinlabel $\gamma^+$ at 465 95
    \pinlabel $L$ at 463 133
     \end{footnotesize}
\endlabellist

\includegraphics[width=0.9\textwidth]{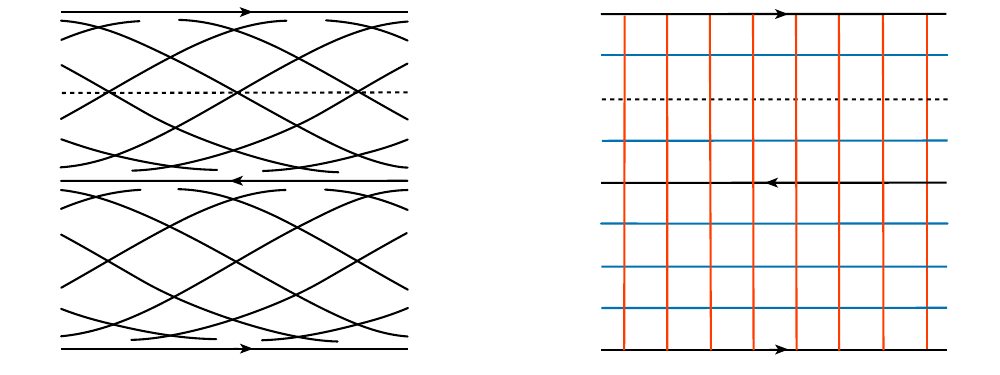}
  \caption{On the left, the foliation on an embedded Birkhoff torus induced by the stable and unstable weak foliations. On the right, the characteristic foliations induced by the bi-contact structure $(\ker\alpha_+,\ker\alpha_-)$. }
\label{fig:birkhoff foliations}
\end{figure}

\begin{subsection}{Structures on a Birkhoff torus and Legendrian knots}
\label{Bir}
Let S$\;=\mathbb{H}^2 /\Gamma$ be a hyperbolic surface and let $\gamma$ be a closed geodesic. The lift of $\gamma$ to the unit tangent bundle of $\gamma$ is an immersed {\it Birkhoff} torus UT$\gamma$ with two closed orbits $\{\gamma_1,\gamma_2\}$, in the sense that UT$ \gamma$ is an immersed torus transverse to the geodesic flow in UT$\gamma\setminus \{\gamma_1,\gamma_2\}$ and tangent to the geodesic flow along $\{\gamma_1,\gamma_2\}$. If the geodesic $\gamma$ is also simple UT$ \gamma$ is an embedded torus.  

We now describe a special knot $L$ on the Birkhoff torus that plays a very important role in numerous applications. This knot is defined by the angle $\theta=\frac{\pi}{2}$ on each fiber along $\gamma$ and has the following remarkable properties (see \fullref{bundles1} and \fullref{fig:birkhoff foliations}).

\begin{enumerate}
\item $L$ is a Legendrian knot for the contact structure $\eta_+$ preserved by the flow,
\item $L$ is transverse to the weak stable and unstable foliations,
\item $L$ is a Legendrian-transverse knot with respect to the bi-contact structure $(\ker\alpha_-,\ker\alpha_+)$.

\end{enumerate}

\end{subsection}

\section{Foulon--Hasselblatt contact surgery} % (fold)
\label{sec:Foulon}

Let $\alpha$ be a $C^k$ contact form with Reeb vector field $R_{\alpha}$,  and consider a Legendrian knot $L$ for $\xi=\ker \alpha$. Foulon and Hasselblatt describe in \cite{FoHa1} a family of contact surgeries of the \textit{Dehn} type that generates a new $C^k$ contact form $\tilde{\alpha}$ in the new manifold $\tilde{M}$ and a $C^{k-1}$ vector field $R_{\tilde{\alpha}}$ such that

\begin{enumerate}

\item $\tilde{\xi}=\ker \tilde{\alpha}$ is isotopic to  a contact structure obtained from $\xi=\ker \alpha$ by classic contact surgery and such that $\tilde{\alpha}=\alpha$ outside a neighborhood of an annulus $C$ transverse to the flow of $R_{\alpha}$.

\item $R_{\tilde{\alpha}}$ is the Reeb vector field of $\tilde{\alpha}$. Moreover $R_{\tilde{\alpha}}$ is collinear to $R_{\alpha}$ in $M\setminus C$ and is a reparametrization of $R_{\alpha}$ in a neighborhood of $C$.  

\end{enumerate}

In other words Foulon and Hasselblatt construction is a contact surgery that allows us to have some control on the direction of the resulting Reeb vector field.

\subsection{Definition of the contact surgery}
\label{sub:contact surgery}

In the following we recall the main features of Foulon-Hasselblatt construction (see \cite{FoHa1} and \cite{FHV2} for more details). Given a 3-manifold with a contact structure $\xi=\ker\alpha$ and a Legendrian knot $L$ there is a coordinate system $$(t,s,w)\in N= (-\delta,\delta)\times S^1 \times (-\epsilon,\epsilon),$$
where the parameters $(s,w)$ are defined on the \textit{surgery annulus} $C=\{0\}\times S^1 \times (-\epsilon,\epsilon)$ (see \fullref{fig:Anosov flow}). More precisely, $s \in S^1$ is the parameter of $L$ and $w$ belongs to some interval $(-\epsilon,\epsilon)$. The transverse parameter $t$ is such that the Reeb vector field of $\gamma$ satisfies $R_\alpha=\frac{\partial}{\partial t}$, therefore $N$ is a flow-box chart for $R_\alpha$. In this coordinate system a contact form defining a contact structure takes the particularly simple expression $\gamma=dt+w\: ds$.

\begin{figure}

\labellist 
     \begin{footnotesize}
    
    \pinlabel $L$ at 475 83
     \end{footnotesize}
\endlabellist

\includegraphics[width=1\textwidth]{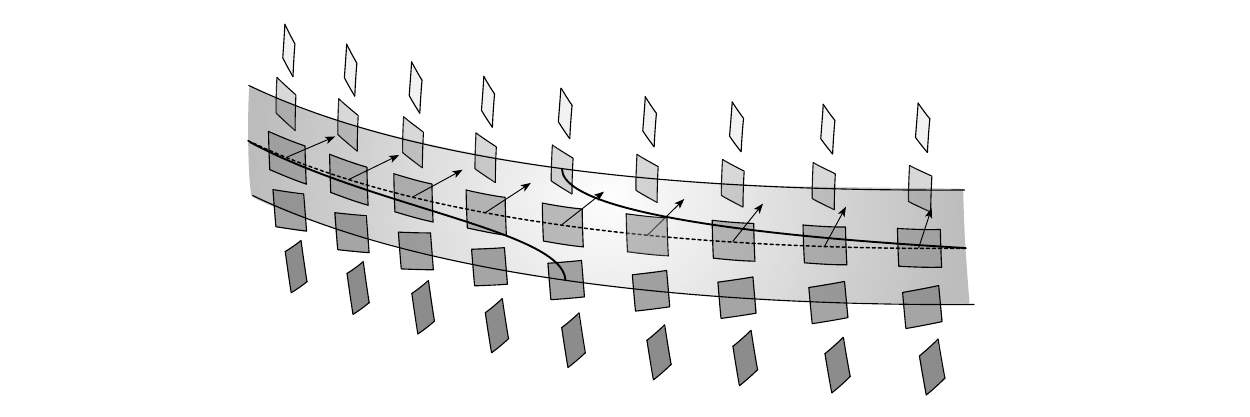}
  \caption{The surgery annulus in Foulon-Hasselblatt construction.}
\label{fig:Anosov flow}
\end{figure}

The surgery can be thought as first cutting the manifold $M$ along the annulus $C$ and then gluing back the two sides of the cut in a different way. In particular, we glue the point that on one side of the cut is described by coordinates $(s,w)$ to the one described by coordinates $(s+g(w),w)$ on the other side. Here $g:[-\epsilon,\epsilon]\rightarrow S^1$ is a non decreasing function such that $g(-\epsilon)=0$ and $g(\epsilon)=2\pi q$ and satisfying a number of additional requirements. The transition map 
$G:C\rightarrow  C,\; (s,w)\rightarrow (s+g(w),w)$
is often called \textit{shear}. Since 
$$G_*\frac{\partial}{\partial t}=\frac{\partial}{\partial t}$$
the shear define a smooth vector field from $R_\alpha$. Note that if we restrict our attention to the vector field $R_{\alpha}$, Foulon-Hasselblatt construction can be interpreted as a Goodman surgery.

Let us denote with $(\alpha)^-$ the contact form defined on one side of the flow-box chart and with $(\alpha)^+$ the contact form on the other side, we easily see that 

$$G^*(\alpha)^+=(\alpha)^-+wg'(w)\:dw
$$

therefore the shear defines a smooth vector field from $R_\alpha$ but not a smooth contact structure on the new manifold. This issue is addressed by Foulon and Hasselblatt introducing a deformation that yields to a 1-form $(\tilde{\alpha})^{+}=(\alpha)^+-dh$ of class $C^1$ and $dh$ is the differential of the following function

$$h(t,w)=\lambda(t) \int_{-\epsilon}^{w}xg'(x)\;dx.$$

Here $\lambda:\mathbb{R}\rightarrow[0,1]$ is $C^1$ bump function with support in some interval $(0,\delta)$ with $\delta>0$ that takes value $1$ in a neighborhood of $0^+$ and takes value $0$ in a neighborhood of $\delta^-$. 
With this choices we have $G^*(\alpha-dh)^+=(\alpha)^-$. Therefore $\tilde{\alpha}$ is a well defined 1-form  of class $C^1$ in $\tilde{M}$.
Note that the deformation  depends on the shear and it is not immediately clear if the plane field distribution $\ker\tilde{\alpha}$ still defines a contact structure. The authors show that if we choose $0<\epsilon< \frac{\delta}{2\pi q}$ this is in fact the case. Moreover the Reeb vector field of $\tilde{\alpha}$ takes the form $$R_{\tilde{\alpha}}=\frac{R_\alpha}{1- dh(R_\alpha)}.$$

\subsection{Application to contact Anosov flows}
Foulon--Hasselblatt construction is purely contact geometric, it does not require the Reeb vector field of $\alpha$ to be Anosov and it can be performed in a neighborhood of any Legendrian knot.

If applied to a geodesic flow on the unit tangent bundle of an hyperbolic surface we have the following.

\begin{theorem}
[Foulon--Hasselblatt--Vaugon \cite{FHV2}]
\label{prop:FH}
Select a closed geodesic $\gamma$ on a hyperbolic surface $S$ and consider the geodesic flow on $UTS$. Consider the knot $L$ defined by the angle $\theta=\frac{\pi}{2}$ on each fiber along $\gamma$. Let $2\epsilon$ be the width of the surgery annulus.

\begin{enumerate}

\item The $(1,q)$-Dehn surgery along $L$ defined in \fullref{sub:contact surgery} does produce an Anosov flow if $q>0$ for every $\epsilon>0$.

\item It does not produce an Anosov if $-q/\epsilon$ is large enough, i.e., if either $q<0$ is fixed and $\epsilon$ is small enough or if $\epsilon>0$ is fixed and $q<0$ with $\left|q\right|$ big enough.

\end{enumerate}
\end{theorem}

The following result shows that the geodesic $\gamma$ can be chosen in such a way that the result of the surgery is hyperbolic

\begin{theorem}[Calegari/Folklore \cite{FoHa1}]
Let $\gamma$ a closed, filling  geodesic in an hyperbolic surface $S$. Then the complement of its image in the unit tangent bundle of $S$ is hyperbolic.
\end{theorem}

Note that the Birkhoff torus corresponding to a closed, filling  geodesic is self intersecting. However, for $q>0$ we can chose $\epsilon>0$ small enough to ensure that the surgery annulus $C_{\epsilon}$ is embedded in $M$.

\section{bi-contact surgery on Anosov flows} 
\label{sec:Local}

We introduce a new type of Dehn surgery on an Anosov flow with weak orientable invariant foliations. Under this assumption the flow is defined by a co-oriented bi-contact structure $(\xi_-,\xi_+)$. Our construction is defined in a neighborhood $N$ of a knot $K$ that is simultaneously Legendrian for $\xi_-$ and transverse for $\xi_+$. Furthermore we require that in $N$ the Reeb vector field of $\alpha_-$ is contained in $\xi_+=\ker \alpha_+ $. As Hozoori shows in \cite{Hoz2}, if the flow is Anosov and volume preserving there is a bi-contact structure $(\ker \alpha_-= \xi_-, \ker \alpha_+=\xi_+)$ such that $R_{\alpha_-}\in \ker \alpha_+=\xi_+$ and $R_{\alpha_+}\in \ker \alpha_-=\xi_-$, therefore the above condition is verified for every Legendrian-transverse knot $K$. On the other hand, if the flow is just Anosov Hozoori \cite{Hoz2} shows that the condition $R_{\alpha_-}\in \ker \alpha_+=\xi_+$ can be always achieved in a neighborhood of a closed orbit. 
\begin{theorem}[Hozoori \cite{Hoz2}]
\label{H1}
Let $\phi^t$ be a $C^1$-Hölder Anosov flow with orientable weak invariant foliations and $\gamma$ a periodic orbit. There is a pair of contact forms $(\alpha_-,\alpha_+)$, such that $(\ker \alpha_-,\ker \alpha_+)$ is a supporting bi--contact structure for $\phi^t$ and, in a neighborhood $N$ of $\gamma$, the Reeb vector field $R_{\alpha_-}$ of $\alpha_-$ satisfies the condition $\alpha_+(R_{\alpha_-})=0$.
\end{theorem}

We use the following to construct a Legendrian-transverse knot from a closed orbit $\gamma$.

\begin{corollary}
\label{Ltpo}
Let $\gamma$ be a closed orbit of a flow defined by a bi-contact structure $(\ker \alpha_- =\xi_-,\ker \alpha_+=\xi_+)$ such that in a neighborhood $N$ of $\gamma$ the Reeb vector field $R_{\alpha_-}$ belongs to $\ker \alpha_+=\xi_+$. The knot $K=\phi^t(\gamma)$ is a Legendrian-transverse knot for $0<t<1$ sufficiently small.
\end{corollary} 

\begin{definition}
Consider a closed orbit $\gamma$ and a pair of contact forms $(\alpha_-,\alpha_+)$ such that in a neighborhood $N$ of $\gamma$ the Reeb vector field $R_{\alpha_-}$ belongs to $\ker \alpha_+=\xi_+$. We call a {\it Legendrian-transverse push-off of $\gamma$} a knot $K$ that is simultaneously Legendrian for $\xi_-=\ker \alpha_-$ and transverse for $\xi_+=\ker \alpha_+$ obtained by translating $\gamma$ using the flow of $R_{\alpha_-}$ (see \fullref{bundles3}).

\end{definition}
\begin{figure}
\label{bundles3}

\includegraphics[width=0.9\textwidth]{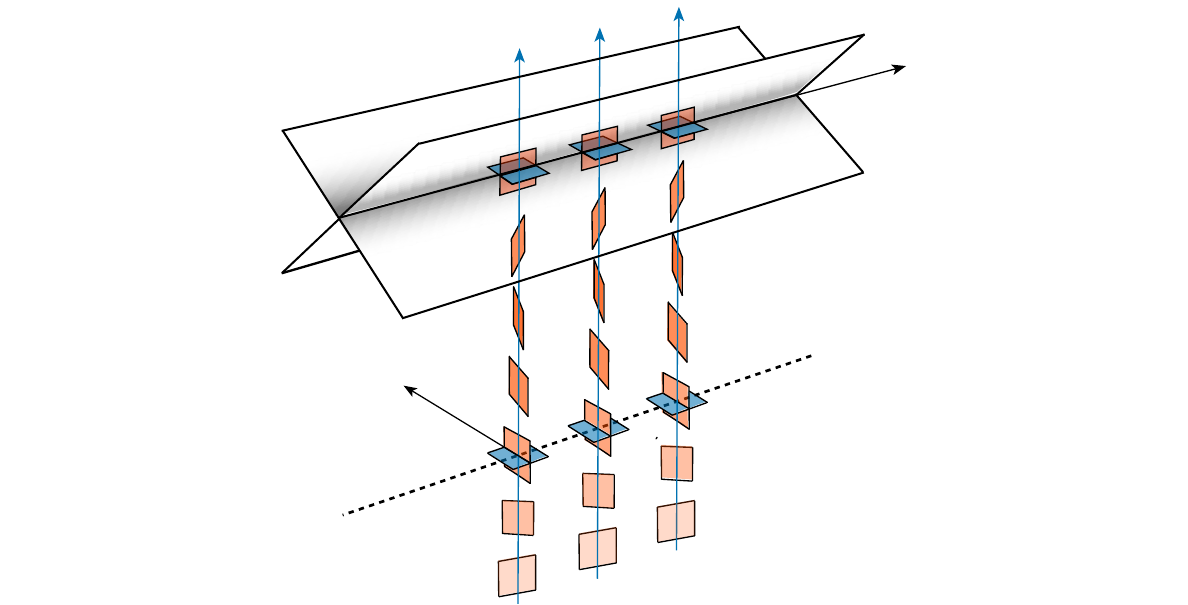}
  \caption{Legendrian-transverse push-off of a closed orbit. In blue $R_{\alpha_- }$. In black, the Anosov flow.}
\label{bundles3}
\end{figure}

\subsubsection{Sketch of the construction of a bi-contact surgery}  Let $K$ be a Legendrian-transverse knot such that in a neighborhood $N$ of $K$ we have $R_{\alpha_-}\in \ker \alpha_+$. Since $K$ is Legendrian for $\xi_-$ we can apply Foulon--Hasselblatt contact surgery to the pair $(\alpha_-,R_{\alpha_-})$ using a shear $F$ on a surgery annulus $A_0$ constructed by translating the knot $K$ using the Anosov flow $\phi^t$. Note that $A_0$ is transverse to $R_{\alpha_-} $ and tangent to $X$. After surgery we have a new contact form $\tilde{\alpha}_-$ with Reeb vector field $R_{\tilde{\alpha}_-}$ that is collinear to $R_{\alpha_-}$ in $M\setminus A_0$. Consider now the contact structure $\ker \alpha_+=\xi_+$ transverse to the knot $K$. If the surgery coefficient $q$ is positive we will show that there is a new contact form $\tilde{\alpha}_+$ such that the Reeb vector field $R_{\tilde{\alpha}_-}$ belongs to $\ker \tilde{\alpha}_+=\tilde{\xi}_+$. 

Suppose now that the initial defining bi-contact structure is such that $R_{\alpha_-}\in \ker \alpha_+$ everywhere (therefore Anosov by \fullref{cond:Anosov}). The new bi-contact structure $(\ker \tilde{\alpha}_-,\ker \tilde{\alpha}_+)$ has the same property, therefore the supported vector field $\tilde{X}$ is Anosov. 
\begin{remark}
The condition $R_{\alpha_-}\in \ker \alpha_+$ {\it everywhere}, can be always achieved in a volume preserving Anosov flow (see \fullref{VPflows}). At the time of writing this work it is not known if the condition can be achieved on any Anosov flow with orientable weak foliations. 
\end{remark}

\subsection{Definition of the bi-contact surgery} We first construct a special coordinate system in a neighborhood $N$ of the Legendrian-transverse knot $K$ where two supporting contact forms $(\alpha_-,\alpha_+)$ take a very simple expression.

\begin{figure}

\includegraphics[width=0.9\textwidth]{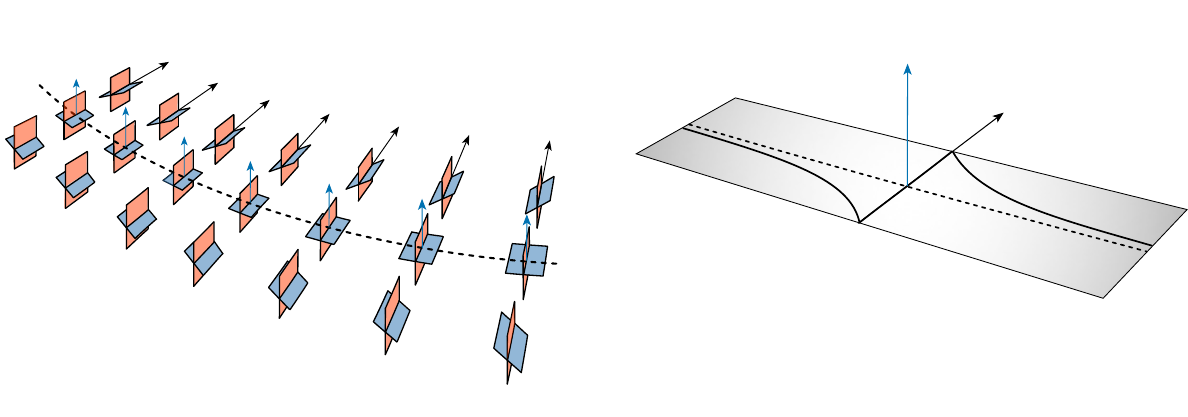}
  \caption{The surgery annulus $A_0$ spanned by the flow. The dashed line is the Legendrian-transverse knot $K$. In blue $R_{\alpha_- }$. In black, the Anosov flow.}
\label{fig:Coordinates0}
\end{figure}
 
 \begin{lemma}
\label{lem:coordinate system}
Suppose $M$ is a 3-manifold endowed with a projectively Anososv flow $\phi^t$ defined by a bi-contact structure $(\xi_-,\xi_+)$ and consider a knot $K$ that is simultaneously Legendrian for $\xi_-$ and transverse for $\xi_+$. Suppose also that there is a contact form $\alpha_-$ defining $\xi_-$ such that the Reeb vector field of $\alpha_-$ is contained in $\xi_+$ in a neighborhood $N$ of $K$.
We can choose $N=A_0\times(-\epsilon,\epsilon)$ equipped with coordinates $(s,v,w)$ and a contact form $\alpha_+$ supporting $\xi_+$ such that in $N$ we have
$$\alpha_-=dw+v\:ds$$
$$\alpha_+=ds-b(s,v,w)\:dv\;\;\;with \;\;\alpha_+=ds \;\;on \;\;A_0 .$$
Here $s$ and $v$ are coordinates on $A_0$ such that $s\in S^1$ is the parameter describing $K$ and $v\in(-\delta,\delta)$ with $\delta\in \mathbb{R}^+$ while $w\in (-\epsilon,\epsilon)$ is a transverse parameter to $A_0$.

\end{lemma}

\begin{proof}
The construction of the surgery annulus $A_0$ is straightforward. We start with the Legendrian-transverse knot $K:S^1\rightarrow M$ and we set the tangent annulus $$A_0=\bigcup_{t\in [-\tau,\tau]}\phi^t(K)$$ spanned by the flow $\phi^t$ for $t\in[-\tau,\tau]$ with $\tau>0$. Let $\alpha_-$ be the 1-form specified in the statement of the lemma. Since $K$ is Legendrian we can choose $\tau$ such that the Reeb vector field $R_{\alpha_-}$ is transverse to $A_0$ (see \fullref{fig:Coordinates0}).

Let $\phi^w$ the flow of the Reeb vector field $R_{\alpha_-}$ of $\alpha_-$. There is an auxiliary coordinate system $(s,u,w)$ in a neighborhood of $A$ such that 
$$\alpha_-=dw+a(s,u)\:ds.$$
Since $K$ is a Legendrian knot, we have $a(s,0)=0$ and since $\frac{\partial}{\partial u}a(s,u)\neq 0$, the transformation of coordinates $$(s,u)\rightarrow (s,a(s,u))=:(s,v)$$
is non singular and therefore we have (see \cite{FoHa1})
$$\alpha_-=dw+v\:ds.$$ 
Since the flow lines of $R_{\alpha_-}$ are Legendrian for $\xi_+$ by hypothesis and the $v$-curves are transverse to $\xi_+$ in a neighborhood $N$ of $K$, there is a contact form $\alpha_+$ supporting $\xi_+$ that in $N$ can be written
$$\alpha_+=ds-b(s,v,w)\:dv.$$
The $v$-curves do not describe the same codimension two foliation described by the flowlines of $X$. However, the foliations coincide on the surgery annulus. Therefore  $b(s,v,0)=0$ and $\alpha_+=ds$.

\end{proof}

\begin{lemma}
\label{LL}
In the hypothesis of \fullref{lem:coordinate system} the contact form $\alpha_+^0=ds-b(s,v,w)\:dv$ can be isotoped along the flow lines of $R_{\alpha_-}$ to a contact form independent on the $s$-coordinate $\alpha_+^1=ds-b(v,w)\:dv$.

\end{lemma}

\begin{proof}

Consider the minimum value $b_{min}(v,w)$ taken by $b(s,v,w)$ along the $s$-coordinate. Fixing $w=\overline{w}$ and $k>0$ small we have $b(s,v,\overline{w}+k)>b_{min}(v,\overline{w})$. Smoothly interpolate $b(s,v,w)$ and $b_{min}(v,w)$ in the interval $(\overline{w},\overline{w}+k)$ and in a neighborhood of $v=\pm \delta$ in such a way that the resulting interpolation $b_{int}(s,w,v)$ is decreasing if $w$ decreases. We define the new $b$ to be $b(s,w,v)$ for $w>\overline{w}+k$, to be $b_{int}(s,w,v)$ in $(\overline{w},\overline{w}+k)$ and to be $b_{min}(v,w)$ in $w<\overline{w}$. In a slightly smaller neighborhood $N'\in N$ we have $b=b_{min}(v,w)$ independent on $s$, therefore
$$\alpha_+^1=ds-b(v,w)\:dv.$$
Note that the above procedure relies on the fact that $b=0$ on the surgery annulus $A_0$.
\end{proof}

\begin{remark} 
If the initial flow $\phi^t$ is defined by a pair of contact forms $(\alpha_-,\alpha_+^0)$ such that $R_{\alpha_-}\in \ker \alpha_+^0$, $\phi^t$ is Anosov there is a path of Anosov flows connecting $\phi^t$ to the flow $\psi^t$ defined by $(\alpha_-,\alpha_+^1)$. In particular $\phi^t$ and $\psi^t$ are orbit equivalent by a diffeomorphism isotopic to the identity (see \fullref{LS}).
\end{remark}

\subsubsection{Deformations and gluings }
\label{glu}
Our procedure can be thought as first cutting the manifold $M$ along the annulus $A_0$ and then gluing back the two sides of the cut in a different way. In particular, we glue the points described by the coordinates $(s,v)$ on the side of the cut $A_0^-$ where $w\leq 0$ to points $(s+f(v),v)$ on $A_0^+$, the side of the cut that bounds  the region where $w\geq 0$. We define 
$$f:[-\delta,\delta]\rightarrow S^1,\;\;\;v\rightarrow f(v),$$ 
where $q\in \mathbb{Z}$ and $f:\mathbb{R}\rightarrow [0,2\pi]$
non increasing in $(-\delta,\delta)$, smooth, and such that $f((-\infty,-\delta))=0$, $f([\delta,\infty))=-2q\pi$. 
\begin{remark}
Note that the function $f$ used here to define the shear is {\it non-increasing} while the function used in \fullref{sec:Foulon} to define Foulon-Hasselblatt construction is non-decreasing. This is because Foulon-Hasselblatt construction was originally performed along a transverse annulus $C$ while the Legendrian-transverse surgery is performed on a an annulus $A_0$ tangent to the flow. With this convention if $K\subset C\cap A_0$ the two operations produce homeomorphic manifolds for the same $q$.
\end{remark}

As noticed in \fullref{sub:contact surgery} the shear $$F:A_0\rightarrow  A_0,\; (s,v)\rightarrow (s+f(v),v)$$
defines a smooth new $3$-manifold $\tilde{M}$
but it does not preserve the contact form $\alpha_-$ since on the surgery annulus $A_0$
$$F^*(\alpha_-)^+=dw+v\:d(s+f(v))=(\alpha_-)^-+vf'(v)\:ds.$$ An analogous calculation shows that $\alpha_+$ is not preserved either since
$$F^*(\alpha_+)^+=F^*ds=d(s+f(v))=ds+f'(v)\:dv=(\alpha_+)^-+f'(v)\:dv.$$
The strategy is to deform each of the contact structures independently on $M\setminus A_0$ in such a way that they remain transverse to each other and after the application of the shear they define two new (transverse) contact structures (see \fullref{fig:Defor3}).

\begin{figure}

\includegraphics[width=1\textwidth]{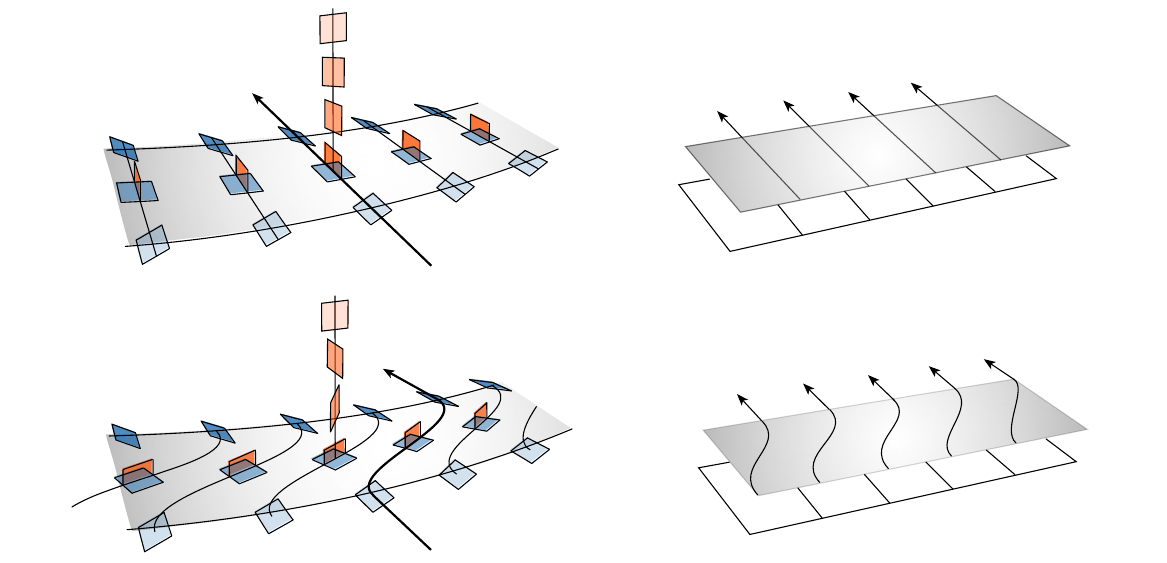}
  \caption{On the top left: the bi-contact structure of the original flow. On the bottom left: the bi-contact structure after deformation. On the right, the corresponding flow-lines on the two sides of the surgery annulus $A_0$.    }
\label{fig:Defor3}
\end{figure}

The deformation that we apply to $\alpha_-$ is the same one used by Foulon and Hasselblatt to define their contact surgery.

 We set $(\tilde{\alpha}_-)^+=(\alpha_-)^+- dh$ where 
 
 $$h(v,w)=\lambda_1(w) \int_{-\delta}^{v}xf'(x)\;dx.$$

Here $\lambda_1:\mathbb{R}:\rightarrow [0,1]$ is a smooth bump function supported in $(0,\epsilon)$ such that $\lambda_1(0)=1$ and $\left| \lambda_1' \right|<\frac{1}{\epsilon}+\iota$ with $\iota$ arbitrarily small. Note that in a neighborhood of $A$
we have \cite{FHV2} $$\tilde{\alpha}_-=dw+v\:ds- f'(v)v\:dv,$$ 
therefore $$F^*(\alpha_--dh)^+=(\alpha_-)^-$$
and we have a well defined 1-form on the new manifold $\tilde{M}$.

The deformation that we apply to $\alpha_+$ is defined as follows. We set $\tilde{\alpha}_+=\alpha_+-\sigma$ and consider the 1-form 
\begin{equation}
\label{deformation}
\sigma=\lambda_2(w)\:f'(v)\:dv
\end{equation}
where $ \lambda_2: \mathbb{R}\rightarrow [0,1] $ is a smooth bump function, with support in $(0,\epsilon)$ that take the value $1$ in a neighborhood of $0^+$ and value $0$ in a neighborhood of $\epsilon^-$ . More explicitly 

$$\tilde{\alpha}_+=ds-(b(v,w)+\lambda_2(w)\:f'(v))\;dv.$$ 
 Note that since $f'(v)$ has support in $A_0$ the 1-form $\sigma$ is smooth with support in the interior of $N=A_0\times [-\epsilon,\epsilon]$.

We check that on the surgery annulus $A_0$ we have $F^*(\alpha_+-\sigma)^+=(\alpha_+)^-$. 
 A direct computation shows that 
$$F^*(\tilde{\alpha}_+)^+=F^*(ds-b\;dv-\lambda_2(w)\:f'(v)\;dv)=ds+f'(v)\;dv-b\;dv-f'(v)\;dv=(\alpha_+)^-$$ therefore $\tilde{\alpha}_+$ is a well defined $1$-form on $\tilde{M}.$

The plane field distribution $\tilde{\xi}_+=ker(\tilde{\alpha}_+)$ and the plane field distribution $\tilde{\xi}_-=ker(\tilde{\alpha}_-)$ are transverse.

\begin{theorem}
\label{bicond}
Let $(\xi_-=\ker \alpha_-,\xi_+=\ker \alpha_+)$ a bi-contact structure such that in a neighborhood of a knot $K$ that is simultaneously Legendrian for $\ker \alpha_-$ and transverse for $\ker \alpha_+$ we have $R_{\alpha_-}\subset \ker \alpha_+$. If $q>0$ a bi-contact $(1,q)$-surgery along $K$ produces a new bi-contact structure $(\tilde{\xi}_-=\ker \tilde{\alpha}_-,\tilde{\xi}_+=\ker \tilde{\alpha}_+)$. Moreover $R_{\tilde{\alpha}_-}\subset \ker \tilde{\alpha}_+$
\end{theorem}

\begin{proof}

We first show that the $(1,q)$-surgery produces a 1-form $\tilde{\alpha}_-$ that is contact for every $q\in \mathbb{Z}$. This is done using the contact surgery introduced by Foulon and Hasselblatt in \cite{FoHa1} and interpreting the surgery annulus $A_0$ as a transverse annulus to the Reeb vector field of $\alpha_-$. We use a slightly different argument to the one used by Foulon and Hasselblatt. In particular we show that we do not need any bound on the derivative of $f$ in the definition of the shear in $\ref{glu}$. As shown in \cite{FoHa1} we have 

$$\tilde{\alpha}\wedge d\tilde{\alpha}=(-1+\frac{\partial h}{\partial w})\:dV,$$ 
with $dV=ds\wedge dv\wedge dw$. Therefore the condition $\left|\frac{\partial h}{\partial w}\right|<1$ ensures that $\tilde{\alpha}$ is contact.

$$\left|\frac{\partial h}{\partial w}\right|=\left|\lambda'_1 \int_{-\delta}^{v}xf'(x)\;dx\right|,$$
integrating by parts and since $f(-\delta)=0$ for every $q\in \mathbb{Z}$ we get $$\left|\frac{\partial h}{\partial w}\right|=\left|\lambda'_1 (vf(v)+\delta f(-\delta)-\int_{-\delta}^{v}f(x)\;dx)\right|= \left|\lambda'_1 (vf(v)-\int_{-\delta}^{v}f(x)\;dx)\right|$$ 
since $\left|\lambda'_1\right|<\frac{1}{\epsilon}+\iota$ with $\iota$ arbitrarily small, $v<\delta$, $\left|f(v)\right|<2\pi \left|q\right|$ and $\int_{-\delta}^{v}f(x)\;dx)< 2\pi \delta $, we have
$$\left|\frac{\partial h}{\partial w}\right|<\frac{1}{\epsilon} \delta 2\pi (\left|q\right|+1)<1.$$
This shows that the 1-form $\tilde{\alpha}_-$ is contact if

\begin{equation}
 \label{6}
0<\delta <\frac{\epsilon}{2\pi (\left| q \right|+1)}.
\end{equation}

\begin{remark}
Note that condition (\ref{6}) holds independently of the sign of $q$ and the positivity of $\alpha_-$. Moreover $\delta$ can be chosen arbitrarily small, in particular the tangent surgery annulus $A_0$ can be chosen arbitrarily thin.
\end{remark}

We now show that for $q>0$, the 1-form $\tilde{\alpha}_+$ defines a contact form on $\tilde{M}$. 
The 1-form $\tilde{\alpha}_+$ defines a positive contact form if and only if \cite{ElTh}

\begin{equation}
\label{anosovity condition}
\frac{\partial b}{\partial w}+ \lambda_2'(w)\:f'(v)>0.
\end{equation}
Since $\frac{\partial b}{\partial w}>0$ by the contact condition and $\lambda_2'(w)\leq 0$ by the definition of the bump function $\lambda_2$, the inequality is always satisfied if $f'(v)\leq 0$. Therefore, the family of $(1,q)$-Dehn surgeries produces bi-contact structures for $q>0$.

\end{proof}

The proof of \fullref{bicond} shows that there is a choice of the direction of the twist that strengthens the contact condition of the contact structure transverse to $K$. Performing the surgery in the opposite direction may result in a violation of the contact condition.
An analogous phenomenon emerges in Goodman surgery (see \cite{HaTh}, \cite{FHV2} and \cite{Bar1}) where there is a choice of the direction of the shear that weaken the hyperbolicity of the flow. The following result shows that there is indeed a strong connection between the Anosovity of the new flows produced by the bi-contact surgery and contact geometry.

\begin{maintheorem}
\label{Anosovity}

Let $M$ be a $3$-manifold endowed with a bi-contact structure $(\ker \alpha_-,\ker \alpha_+)$ such that $R_{\alpha_-}\in \ker \alpha_+$ everywhere. If the Legendrian-transverse surgery yields a new bi-contact structure it yields an Anosov flow.
\end{maintheorem}

\begin{proof}
The proof is a straightforward application of Hoozori criterion of Anosovity. Since by construction we have $R_{\tilde{\alpha} _-}\subset \tilde{\xi}_+$ the Reeb vector field $R_{\tilde{\alpha}}$ is dynamically positive everywhere (\fullref{cond:Anosov}).
\end{proof}

\begin{figure}
\labellist 
     \begin{footnotesize}
     \pinlabel $A_0$ at 170 45
     \pinlabel $A_{\epsilon}$ at 170 185
    \pinlabel $K$ at 80 69
    \pinlabel $C_{out}$ at 50 140
    \pinlabel $C_{in}$ at 140 43
     \pinlabel $\phi^w(K)$ at 107 174
     \end{footnotesize}
\endlabellist
\includegraphics[width=1\textwidth]{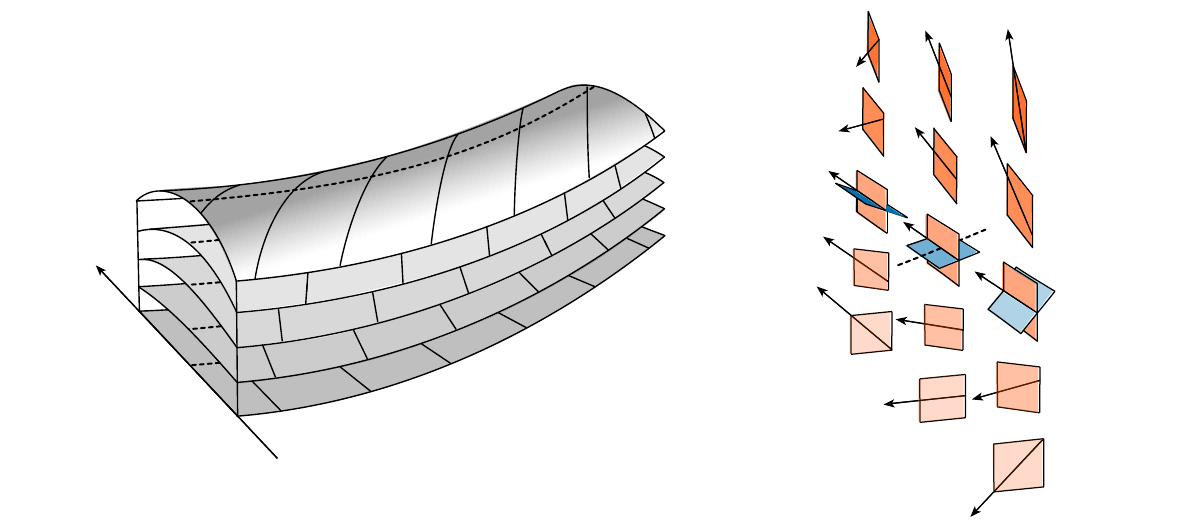}
  \caption{On the right: the red contact structure rotates along the $w$ axis. As usual, the Legendrian-transverse knot is denoted by a dashed line.
On the left: the annuli $A_w$ spanned by the flow.}
\label{fig:Coordinates4}
\end{figure}

\section{New Anosov flows for negative surgery coefficients} 
\label{sec:New contact flows}

We now give a sufficient condition for a bi-contact $(1,q)$-surgery to yield an Anosov flow when the operation is performed in the direction that weaken the hyperbolicity. The condition is purely contact geometric and involves the behaviour of the bi-contact structure along the boundary of a neighborhood $N$ of a Legendrian-transverse knot $K$. We remark that the hypothesis of the next theorem can be always achieved in a volume preserving Anosov flow. 

\begin{theorem}
\label{SLOPE}
Let $K$ be a Legendrian-transverse knot $K$ in a bi-contact structure $(\xi_-=\ker \alpha_-,\xi_+=\ker \alpha_+) $ defining an Anosov flow, such that $R_{\alpha_-}\subset \xi_+$ everywhere. There is a neighborhood $N$ of $K$ such that if $k$ is the slope of the characteristic foliation induced by $\xi_+$ on $\partial N$ a bi-contavt $(1,q)$-surgery with $q>-k$ yields an Anosov flow.  

\end{theorem}

\begin{proof}

Consider a flow-box neighborhood $N$ of $K$ bounded by the surgery annulus $A=A_0$, 
the transverse annuli $C_{in}$ and $C_{out}$ constructed by flowing the boundary components of $A_0$ using the flow of $R_{\alpha_-}$, and on the top, by an annulus (see \fullref{fig:Coordinates4}).

$$A_{\epsilon}=\bigcup_{t\in [-\tau,\tau]}\phi^t(\phi^{\epsilon}(K)).$$ 
 
After an orbit equivalence isotopic to the identity \footnote{Using a slightly different definition of $k$ we can avoid to perform this orbit equivalence.} as in \fullref{LL} we assume that the expression of $\alpha_+$ in lemma does not depend on the $s$-coordinate. Therefore the characteristic foliation $\mathcal{G}$ induced by $\xi_+$ on $\partial N$ does not depend on the $s$-coordinate. Moreover, without loss of generality, we assume that $\mathcal{G}$ is composed by closed curves homotopic to a curve $r\mu+sK$ where $\mu$ is a meridian of $\partial N$. We define the slope of the characteristic foliation $\mathcal{G}$ as $k:=\frac{s}{r}>0$. Let $q$ be the integral part of $-k$ if $-k$ is not an integer and set $q=-k+1$ if $-k$ is an integer. Since $q>-k$ there is a family of curves $f_q(v):[-\delta,\delta]\rightarrow S^1$ on $A_0$ with $f'_q(-\delta)=f'_q(\delta)=0$ such that its projection $\pi_w(f_s)$ on $A_{\epsilon}$ along the $w$-curves positively intersects the curves of the characteristic foliation $\mathcal{G}|_{\A_0}$ obtained by restricting $\mathcal{G}$ to $A_{\epsilon}$ and such that the foliation on $\partial N$ constructed by replacing $\mathcal{G}|_{\A_0}$ with the curves $\pi_w(f_q(v))$ has slope $q$. The shear $$F:A_0\rightarrow A_0,\;\;\;(s,v)\rightarrow (s+f_q(v),v)$$ induces a $(1,q)$-Dehn surgery.  After gluing, we obtain a new neighborhood $N$ with coordinates $(s,v,w)$ such that the projection along the $w$-curves of the characteristic foliation induced on $A_{\epsilon}$ by $\xi_+$ positively intesects the characteristic foliation induced by $\xi_+$ on $A_0$. Therefore we can extend the $C^1$-contact structure $\xi_+$ defined on $M\setminus N$ to a $C^1$-contact structure $\tilde{\xi}_+$ on $M$ whose planes positively rotate along the $w$-curves.  
Since the $w$-curves are flow lines of $R_{\alpha_-}$, Anosovity follows by Hozoori's criterion (see \fullref{fig:Coordinates5}.)

\end{proof}

\begin{corollary}
In a neighborhood of a closed orbit $\gamma$ of a flow defined by a bi-contact structure $(\xi_-=\ker \alpha_-,\xi_+=\ker \alpha_+) $ such that $R_{\alpha_-}\subset \xi_+$ everywhere, there is a fixed tangent annulus $A_0$  such that for every $q\in \mathbb{Z}$ a Legendrian-transverse $(1,q)$-surgery yields an Anosov flow for every $q\in \mathbb{Z}$.  
\end{corollary}

\begin{proof}
We have to prove the statement just for $q<0$. By \fullref{VPflows} there is a bi-contact structure $(\ker \alpha_-,\ker \alpha_+)$ such that $R_{\alpha_-}\in \ker \alpha_+$. Therefore there is a knot $K$ that is a Legendrian-transverse push-off of $\gamma$. 
Let $\epsilon_{\gamma}$ be the value of the parameter $w$ defined by the flow of $R_\alpha$ along $\gamma$. Construct a flow-box neighborhood as in \fullref{SLOPE}. For $\epsilon$ approaching the value  $\epsilon_{\gamma}$ we have  $k\rightarrow \infty$ for a fixed $A_0$. Therefore for every $k>0$ there is a flow-box neigborhood $N_k$ as in \fullref{SLOPE} bounded below by the annulus $A_0$.
\end{proof}

\begin{figure}

\includegraphics[width=1\textwidth]{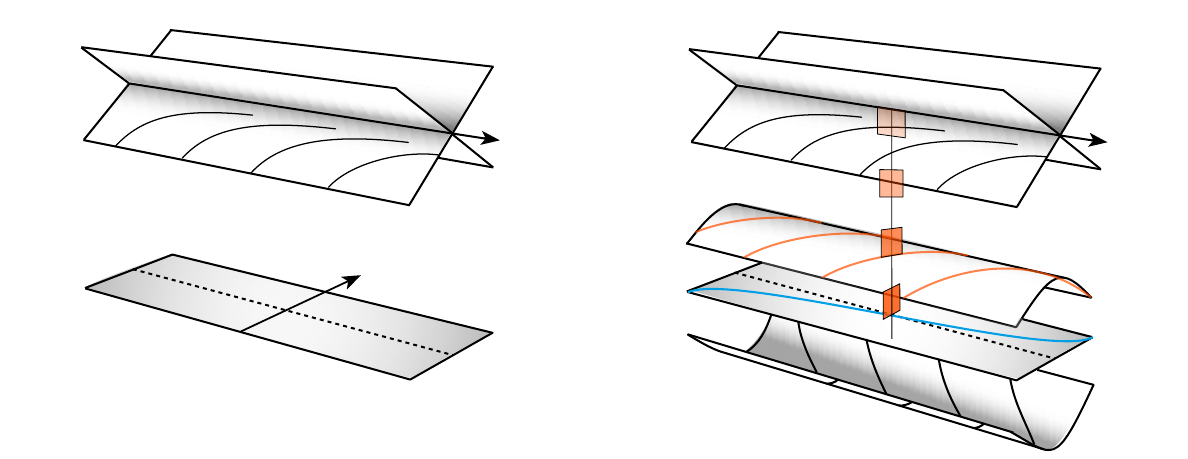}
  \caption{In orange, the contact structure field $\xi_+$ in the proximity of a closed orbit. The orange line are the leaves of the characteristic foliation induced by $\xi_+$ on $A_{\epsilon}$.}
\label{fig:Coordinates5}
\end{figure}

\begin{remark}
\label{Rem}
In a neighborhood $N$ of a closed orbit $\gamma$ of a flow defined by a bi-contact structure $(\xi_-=\ker \alpha_-,\xi_+=\ker \alpha_+) $ such that $R_{\alpha_-}\subset \xi_+$, there is a sequence of nested neighborhoods of the Legendrian-transverse knot $N_{-1}\subset N_{-2}\subset \cdots\subset N_q\subset \cdots\subset N_{-\infty}$ where $$N_{q}=\bigcup_{w\in [0,\epsilon_q]}A_{w},$$
where $$A_{w}=\bigcup_{t\in [-\tau,\tau]}\phi^t(\phi^{w}(K))$$
where $\epsilon_q$ is a sufficiently close to $\epsilon_{\gamma}$. Moreover
$$N_{-\infty}=\overline{\bigcup_{w\in [0,\epsilon_{\gamma})}A_{w}}.$$ 
All the neighborhood are bounded on one side by $A_0$. For $q\in \mathbb{Z}^-$ the neighborhood $N_q$ is a flow-box (see \fullref{fig:Coordinates4} and \fullref{bundle7}).

\end{remark}

\section{Relations with Goodman surgery} % (fold)
\label{sec:Relations}

In this chapter we study the relationship between the bi-contact surgery and the surgery introduced by Goodman \cite{Goo}.

At a first glance the two operations look really different. First of all the surgery annulus used in the bi-contract surgery is tangent to the generating vector field $X$ while in Goodman's construction we use an annulus transverse to the flow. Secondly, in the bi-contact surgery construction the deformation of the bi-contact structure induces a $3$-dimensional perturbation of the flow in a neighborhood of the surgery annulus. On the other hand, in Goodman's surgery the flow lines of the new flow coincide to the old ones outside the (transverse) surgery annulus. Despite these differences there are always choices such that the flows generated by the bi-contact surgery and the ones generated by Goodman's construction are orbit equivalent.

\begin{figure}
\labellist 
     \begin{footnotesize}
     \pinlabel $A_0$ at 130 35
     \pinlabel  $C_{out}$ at 192 80
     \pinlabel $v$ at 220 35
    \pinlabel $w$ at 67 110
    \pinlabel $A_{\epsilon}$ at 100 85
    \pinlabel $C_{in}$ at 40 70
      \pinlabel $s$ at 210 10
     \end{footnotesize}
\endlabellist
\includegraphics[width=1\textwidth]{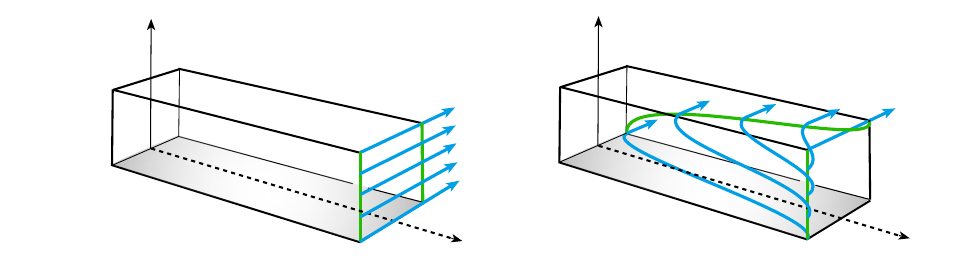}
  \caption{The flow-box $N$. For simplicity we depicted the top layer $A_{\epsilon}$ as flat.  The two vertical green segment in the left picture represent two leaves of $\mathcal{F}$ while on the right are represented two leaves of the new foliation $\mathcal{\tilde{F}}$. A negative Dehn twist on $A_0$ (in blue) corresponds to a positive Dehn twist on $C_{out}$ (in green). Therefore, a flow constructed by a $(1,q)$-bi-contact surgery along $A_0$ is orbit equivalent to a flow generated by $(1,q)$-Goodman surgery along $C_{out}$. }
\label{fig:Coordinates3}
\end{figure}

\begin{theorem}
\label{equivalence}

Suppose that $K$ is a Legendrian-transverse knot in a bi-contact structure $(\xi_-,\xi_+)$ defined by two contact forms $(\alpha_-,\alpha_+)$ such that in a neighborhood $N$ of $K$ the Reeb vector field $R_{\alpha_-}$ is contained in $\xi_+$. A flow generated by a bi-contact $(1,q)$-surgery along the tangent annulus $A_0$ is orbit equivalent to a flow generated by a $(1,q)$-Goodman surgery along a transverse annulus $C$. 

\end{theorem}
\begin{proof}

Let $N$ be the flow-box bounded by the tangent surgery annulus $A_0$, the annulus $$A_{\epsilon}=\bigcup_{t\in [-\tau,\tau]}\phi^t(\phi^{\epsilon}(K))$$ where $\epsilon>0$ and the transverse annuli $C_{in}$ and $C_{out}$ constructed by flowing the boundary components of $A_0$ using the flow of $R_{\alpha_-}$.

Set $C=C_{out}$ and let $\mathcal{F}$ be the foliation obtained by projecting on $C=C_{out}$ the $w$-curves (the flow-lines of $R_{\alpha_-}$) defined on $C_{in}$ using the flow of $X$. Consider a bi-contact $(1,q)$-surgery along the tangent annulus $A_0$ with shear $$F:A_0\rightarrow A_0,\;\;\;(v,s)\rightarrow (v,s+f(v))$$ with $f(-\delta)=0$ and $f(\delta)=-2\pi q$. Let $\tilde{X}$ be the new vector field. On $A_0$ the vector field $X$ induces a foliation directed by the vector $\frac{\partial}{\partial v}$ while $\tilde{X}$ induces a foliation directed by $ f'(v)\:\frac{\partial}{\partial s}+\frac{\partial}{\partial v}$. The new foliation on $A_0$ is obtained from the one induced by $X$ by adding a full $(1,q)$-Dehn twist to the curves of the original foliation. Consequently the curves of the foliation $\mathcal{\tilde{F}}$ of $C_{out}$ constructed by projecting the $w$-curves defined on $C_{in}$ using the flow of $\tilde{X}$ can be obtained from the curves of $\mathcal{F}$ by adding a full $(1,-q)$-Dehn twist (see \fullref{fig:Coordinates3}). An orbit equivalent flow is constructed by cutting $M$ along $C=C_{out}$ and gluing back $C^-$ to $C^+$ using some shear $G$ defined by 
$$G:C^-\rightarrow C^+,\;\;\;(w,s)\rightarrow (w,s+g_s(w))$$
with $g_s(0)=0$ and $g_s(\epsilon)=2\pi q$. 
In general $g_s:[0,\epsilon]\rightarrow S^1$ depends on the coordinate $s$.
\end{proof}

\begin{corollary}
\label{q>-k}
Suppose that $K$ is a Legendrian-transverse knot in a bi-contact structure $(\xi_-=\ker \alpha_-,\xi_+=\ker \alpha_+)$ defined by two contact forms such that $R_{\alpha_-}\in \xi_+$ everywhere. There is a transverse embedded annulus $C_q$ such that a $(1,q)$-Goodman surgery on $C_q$ with $q>-k$ yields an Anosov flow.
\end{corollary}

\begin{corollary}
In the hypothesis of \fullref{q>-k} there is a sequence of nested annuli  $C_{-1}\subset C_{-2}\subset \cdots\subset C_q\subset \cdots\subset C_{-\infty}$ bounded on one side by the the Legendrian-transverse knot $K$. Every $C_q$ is the is the core of $N_q$ and it is transverse to the flow if $q\in \mathbb{Z}$. The annulus $C_{-\infty}$ {\it quasi-transverse} and it is bounded by $K$ on one side and by the closed orbit $\gamma$ on the other side (see \fullref{bundles8}).

\end{corollary}
\begin{figure}
\label{bundles8}

\includegraphics[width=0.9\textwidth]{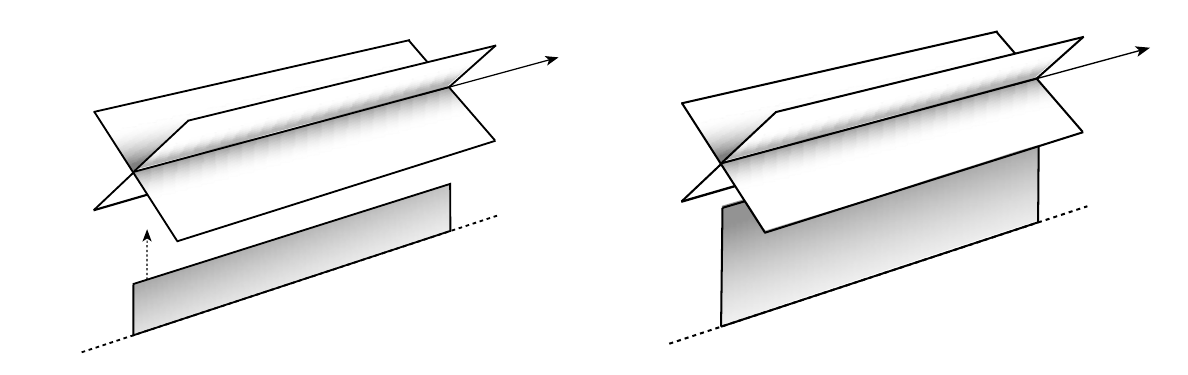}
  \caption{If a $(1,q)$-Goodman surgery along an annulus $C_q$ produces an Anosov flow for a fixed $q<0$, a $(1,q-1)$-Goodman produces an Anosov flow on an annulus $C_{q-1}$ obtaine expanding $C_q$ towards the closed orbit.  }
\label{bundles8}
\end{figure}

\section{Flexibility of bi-contact structures and structural stability}
The interplay between structural stability of an Anosov flow and the flexibility of its underlying bi-contact structure can be used to construct $C^1$ paths of Anosov flows. The idea is roughly as follows. Hozoori's criterion (\fullref{prop:Hoz}) gives a necessary and sufficient condition for a bi-contact structure to define an Anosov flow. In general this condition is not easy to control since it does require the knowledge of the invariant directions in the normal bundle $TM/\langle X \rangle$. Under particular circumstances this information is not necessary. For example we have the following sufficient condition of Anosovity that relies just on the Reeb dinamics of an underlying bi-contact structure and implies \fullref{prop:Hoz}. 

\begin{corollary}
\label{cond:Anosov}
Let $X$ be a projectively Anosov flow and let $(\ker \alpha_-=\xi_-,\ker \alpha_+=\xi_+)$ be a bi-contact structure defining $X$ such that the Reeb vector field $R_{\alpha_-}$ is contained in $\xi_+$. Then $X$ is an Anosov vector field. 
\end{corollary}

We will systematically use the following consequence of \fullref{cond:Anosov}.

 \begin{lemma}
 \label{LS}
Suppose that $M$ is equipped with a pair of contact structures $(\ker \alpha_-=\xi_-,\xi_+)$ defining an Anosov flow $\phi^t$ such that the Reeb vector field $R_{\alpha_-}$ of $\alpha_-$ belongs to $\xi_+$. An isotopy of $\xi_+$ along the flowlines of $R_{\alpha_-}$ defines a path of orbit equivalent Anosov flows.
 \end{lemma}

\begin{proof}
Assume that $R_{\alpha_-}\in \xi_+$. An isotopy of $\xi_+=\ker \alpha_+$ as above defines a path of bi-contact structures $(\xi_-,\ker (\alpha_+)_t)_{t\in [0,1]}$ such that $R_{\alpha_-}\in (\xi_+)_t=\ker (\alpha_+)_t$ in $N$. The statement follows from \fullref{cond:Anosov} and $C^1$-structural stability of Anosov flows.   
\end{proof}

If a flow is Anosov and volume preserving we can ensure the existence of a bi-contact structure satisfying the hypothesis of \fullref{lem:coordinate system}. Using the equivalence of \fullref{equivalence} we can construct transverse annuli along which a $(1,q)$-Goodman surgery (for some $q\in \mathbb{Z}$) produces an Anosov flow. Under these hypothesis we have the following consequence of \fullref{LS}.

\begin{corollary} 
\label{cor:4}
Suppose that $C_{q}$ is an annulus with positive preferred direction transverse to a volume preserving Anosov flow $\phi^t$. Suppose also that a $(1,q)$-Goodman surgery along $C_q$ produces an Anosov flow for a fixed $q<0$.
\begin{enumerate}
\item There is an arbitrarily thin annulus $C\subset C_{q}$ and a flow $\phi_t'$ orbit equivalent to $\phi_t$ such that a $(1,q)$-Goodman surgery along $C$ generates an Anosov flow. 
\item There is a $C^1$-path of Anosov flows connecting $\phi_t$ and $\phi_t'$.
\item The flow-lines of $\phi'_t$ coincide with the ones of $\phi_t$ outside an arbitrarily small flow-box neighborhood $N$ of $C_{q}$.
\end{enumerate}
\end{corollary}

\section{Surgery along simple closed geodesics in a geodesic flow}
\label{sec:Contact}

Let $X$ be the generating vector field of the geodesic flow on the unit tangent bundle of a hyperbolic surface $S$ and let $\beta_+$ the contact form preserved by $X=R_{\beta_+}$. Select a closed geodesic $\gamma$ (eventually self intersecting) on $S$ and consider the knot $L$ defined by the angle $\theta=\frac{\pi}{2}$ on each fiber along $\gamma$. $L$ is a Legendrian knot for $\ker \gamma$. Let $C_{\epsilon}=S^1\times [-\epsilon, \epsilon]\subset$ UT$\gamma$ be a transverse annulus with $\epsilon$ be small enough to ensure that $C_{\epsilon}$ is not self intersecting. As shown by Foulon and Hasselblatt for any $q>0$ their construction ensure that a $(1,q)$-Goodman surgery produces a contact Anosov flow.

If $L$ is associated to a simple closed geodesic $\gamma$ a recent result of Marty \cite{Ma} shows that Goodman surgery produces a skewed $\mathbb{R}$-covered Anosov flow regardless of the sign of the surgery if and only if the closed orbit is the lift of a simple closed geodesic. We now prove a counterpart of Marty's result in the contact category.

To this hand we prove the following lemma showing that while Foulon and Hasselblatt construction requires an arbitarily thin annulus in order to produce a contact flow for positive surgery coefficients, in the case of negative surgery coefficients this requirement is not necessary.

\begin{lemma}
\label{FHn}
Suppose that $R_{\beta_+}$ is the Reeb vector of a  contact form $\beta_+$ defining a positive contact structure and let $L$ be a Legendrian knot for $\ker \beta_+$. For $q<0$ Foulon and Hasselblatt construction along a transverse embedded annulus $C$ produces a new contact form $\tilde{\beta}_+$ with Reeb vector field $R_{\tilde{\beta}_+}$ regardless of the thickness of the surgery annulus.
\end{lemma}

\begin{proof} Foulon and Hasselblatt show that the vector field $\tilde{X}$ obtained by Goodman surgery with shear $$G:C\rightarrow C,\;\;\;(s,w)\rightarrow (w,s+g(w)),$$ along $C$ preserves a positive contact form  
$$\tilde{\beta}_+=dt+w\:ds- dh$$
with 
$$h(t,w)= \lambda(t) \int_{-\epsilon}^{w}xg'(x)\;dx,$$ hence
$$\tilde{\beta}_+=dt+w\:ds-\lambda'(t)\int_{-\epsilon}^{w}xg'(x)\;dx\;dt-\lambda(t)wg'(w)\;dw.$$
They indeed show that the contact condition 
\begin{equation}
\label{ccond}
\tilde{\beta}_+\wedge d\tilde{\beta}_+=(1-\frac{\partial h}{\partial t})\:dV>0
\end{equation} is satisfied regardless of the sign of the Dehn twist if the surgery annulus is sufficiently thin. 
Note that if \begin{equation}
\label{cc2}
\frac{\partial h}{\partial t}=\lambda'(t)\int_{-\epsilon}^{w}xg'(x)\;dx\;dt\leq 0
\end{equation} the contact condition (\ref{ccond}) is satisfied independently of the thickness of the surgery annulus. If $g'(w)<0$, we have $\int_{-\epsilon}^{w}xg'(x)\;dx\geq0$ for $w\in(-\epsilon, \epsilon)$ and since since $\lambda'(t)\leq 0$ for $t>0$
inequality (\ref{cc2}) is satisfied. This corresponds to  $(1,q)$-Foulon-Hasselblatt surgery with $q<0$.

\end{proof}

\begin{lemma}
\label{Lambda}
In a neighborhood $\Lambda$ of a Birkhoff torus associated to a simple closed geodesic there is a coordinate system $(w,s,v)$ such that the  natural contact forms $( \alpha_-,\alpha_+,\beta_+)$ have the following expressions:
$$\alpha_-=dw+v\:ds$$
$$\alpha_+=e^{\frac{1}{2}v^2}(\cos w\:ds-\sin w\:dv)$$
$$\beta_+=e^{\frac{1}{2}v^2}(-\sin w\:ds-\cos w\:dv).$$
\end{lemma}

\begin{proof}
An elementary calculation gives the following expressions for the Reeb vector fields in $N$:
$$V=R_{\alpha_-}=\frac{\partial}{\partial w}$$
$$H=R_{\alpha_+}=e^{-\frac{1}{2}v^2}(\cos w\:\frac{\partial}{\partial s}-\sin w\:\frac{\partial}{\partial v}-v\cos w \:\frac{\partial}{\partial w}) $$ 

$$X=R_{\beta_+}=e^{-\frac{1}{2}v^2}(-\sin w\:\frac{\partial}{\partial s}-\cos w\:\frac{\partial}{\partial v}+v\sin w \:\frac{\partial}{\partial w}). $$ 
The vector field $V$ is periodic in $\Lambda$ and $$\alpha_+(V)=\alpha_-(H)=\beta_+(V)=\beta_+(H)=0.$$

\end{proof}

\begin{figure}
\label{bundle7}
\labellist 
     \begin{footnotesize}
     \pinlabel $L$ at 55 174
     \pinlabel $C_{out}$ at 550 100
    \pinlabel $A_0$ at 50 144
      \pinlabel $w$ at 520 150
      \pinlabel $A_{\epsilon}$ at 520 194
      \pinlabel $A_{-\epsilon}$ at 525 40
     \end{footnotesize}
\endlabellist
\includegraphics[width=0.9\textwidth]{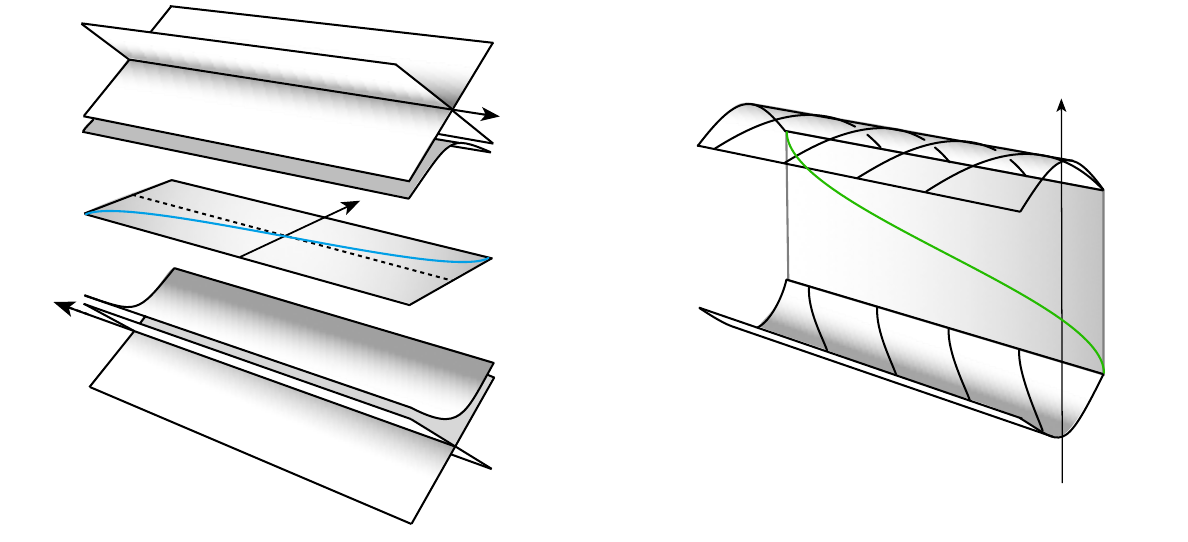}
  \caption{On the left, a neighborhood of a Legendrian knot $L$ associated to a simple closed geodesic on a geodesic flow. The blue curve is the Dehn twist corrsponding to a (negative) Legendrian-transverse surgery (note that the curve has positive slope). On the left: the characteristic foliations induced by $\xi_+$ on $A_{\epsilon}$ and $A_{-\epsilon}$. We perform a Legendrian transverse surgery on $A_0$ "splitting the difference" as in \cite{FoHa1}. This method consists in introducing half of the deformation of the bi-contact structure on the side of the annulus where $w<0$ and half on the side of the annulus where $w>0$. The green curve on $C_{out}$ is the Dehn twist corresonding to the induced Goodman surgery.   }
\label{bundle7}
\end{figure}

\begin{theorem}
Let $C\subset$ UT$\gamma$ be a quasi-transverse annulus associated to a simple closed geodesic. For every $q<0$ there is an embedded transverse annulus $C_q\subset C$ centered in $L$ such that a $(1,q)$-Goodman surgery produces a contact Anosov flow.

\end{theorem}

\begin{proof}
We first show that, for every $q\in \mathbb{Z}$, there are choices such that the bi-contact surgery along a tangent annulus $A_0$ containing $L$ produces an Anosov flow. Then we show that these flows are orbit equivalent to flows constructed by $(1,q)$-Foulon-Hasselblatt surgery. 

Consider the $(s,v,w)$-coordinate system in a neighborhood $\Lambda$ of UT$\gamma$ as in \fullref{Lambda}. Let $N$ be a flow-box neighborhood of $L$  bounded by the tangent surgery annulus $A=A_0$, the annulus $$A_{\epsilon}=\bigcup_{t\in [-\tau,\tau]}\phi^t(\phi^{\epsilon}(K))$$ where $\epsilon>0$ and the transverse annuli $C_{in}$ and $C_{out}$ constructed by flowing the boundary components of $A_0$ using the flow of $R_{\alpha_-}$.  
Since by \fullref{Lambda} $\xi_+$ does not depend on the $s$-coordinate, the characteristic foliation induced by $\xi_+$ on $\partial N$ does not depend on the $s$-coordinate. By \fullref{Rem} for every $q\in \mathbb{Z}$ there is a neighborhood $N_q$ and a new pair of contact structures $(\tilde{\xi}_-,\tilde{\xi}_+)_q$ on the new manifold $\tilde{M}_q$ defining an Anosov flow. Note that $N_q$ and the pair $(\tilde{\alpha}_-,\tilde{\alpha}_+)_q$ can be chosen to be independent on the $s$-coordinate.  
By \fullref{equivalence} these flows are orbit equivalent to flows obtained by $(1,q)$-Goodman surgery with shear $$G:C_{out}\rightarrow C_{out},\;\;\;(s,w)\rightarrow (w,s+g(w))$$ where we chose $g'(w)<0$. Since $L$ is Legendrian for $\ker \beta_+$ and $\beta_+$ is also not dependent on $s$ and rotates along the $w$-coordinate, we can apply Foulon-Hasselblatt construction with surgery annulus $C_{out}$ and coordinates $(s,w,t)$ where $(s,w)$ are defined on $C_{out}$ and $t$ is the parameter given by the geodesic flow (the Reeb flow of $\beta_+$). By \fullref{FHn} the resulting flow is contact regardless of the thickness of $C_{out}$.
Therefore for every $q<0$ the resulting flow is Anosov and contact (see \fullref{bundle7}). 

\end{proof}

\section{Skewed $\mathbb{R}$-covered Anosov flows generated by surgery are contact }

In this chapter we prove a version of Conjecure \fullref{conj1} for surgeries on closed orbit of a geodesic flow on the unit tangent bundle of an hyperbolic surface $S$. This result is a consequence of \fullref{FHn} and work of Asaoka, Bonatti and Marty \cite{ABM} that we recall in the following. 
\subsection{Partial sections, multiplicity and linking numbers.}
Let $M$ be an oriented 3-manifold equipped with a smooth $\phi^t$ flow. We call a subset $P$ of $M$ a {\it partial section} if it is the image of a smooth immersion $\iota:\hat{P}\rightarrow M$ where $\hat{P}$ is a compact surface and the restriction of $\iota$ to the interior of $\hat{P}$ is an embedding transverse to the flow and $\iota(\partial \hat{P})$ is a finite union of closed orbits of $\phi^t$. The immersion $\iota$ lifts to the manifold $M_{\partial P}$ obtained by blowing up along $\iota(\partial \hat{P})=\partial P$. A closed orbit in $\partial P$ is the image $\iota(\hat{\gamma})$ of a boundary component $\hat{\gamma}$ of $\partial \hat{P}$. We denote the lift of $\iota(\hat{\gamma})$ to the blow-up with $\gamma^*$ and we call it {\it boundary component immersed in $\gamma$}.  When the invariant foliations of the Anosov flow are orientable, an immersed boundary component $\gamma^*$ has two invariants: the multiplicity $mult(\gamma^*)$ and the linking number $link(\gamma^*)$ defined as follows\
$$\gamma^*=mult(\gamma^*)\:\lambda_{\gamma}+link(\gamma^*)\:\mu_\gamma$$

Here $\lambda_{\gamma}$ is a curve homotopic to a lift of $\gamma$ in the blow-up manifold (a parallel) while $\mu_{\gamma}$ is a curve homotopic to a fiber of the projection $\pi_{\partial P}:M_{\partial P}\rightarrow M$ (a meridian). If a closed orbit $\gamma$ is the image of just one  boundary components of $\hat{P}$ we can write $mult(\gamma^*)=mult(\gamma)$ and $link(\gamma^*)=link(\gamma)$. We say that a boundary component $\gamma^*$ is {\it positive (negative)} if $mult(\gamma^*)>0$ ($mult(\gamma^*)<0$). A partial section $P$ is said to be positive (negative) when {\it all} its boundary components are positive (negative).

\begin{theorem}[Asaoka--Bonatti--Marty \cite{ABM}]
\label{t1}
If $\phi^t$ is positively skewed $\mathbb{R}$-covered it does not admit a negative partial section.

\end{theorem} 

In \cite{ABM} the authors also study how the multiplicity of a boundary component of a partial section varies after a $(1,q)$-Goodman surgery.

\begin{proposition}[Asaoka--Bonatti--Marty \cite{ABM}]
\label{t2}
A $(1,q)$-Goodman surgery along a closed orbit in $\partial P$ increases the multiplicity of a boundary component $\gamma^*$ by $+
q\: link(\gamma^*)$.
\end{proposition}

\begin{example}
\label{ex2}
Given a non-simple closed geodesic we have an associated Birkhoff annulus $C$ with self intersection. There is a process called {\it Fried desingularization} (see \cite{Frid}) that allows us to obtain from $C$ an immersed partial section $P$ as defined above. The boundary components $\gamma_1$ and $\gamma_2$ of a partial section $C$ associated to a (simple or non-simple) closed geodesic have both multiplicity $1$ and $link(\gamma^*)=p$ where $p$ is the number of self intersections of the geodesic. By \fullref{t2} a $(1,q)$-Goodman surgery with $q<0$ along $\gamma_1$ changes the positivity of only one boundary component. The resulting partial section has two boundary components with opposite sign.

\end{example}
\label{ex1}
\begin{example} In \cite{Ma} (Section 1.5.20) Marty describe a method to produce partial sections associated to a closed orbit $\gamma$ that is the lift of a non-simple closed geodesics on an orientable surface of negative curvature. These partial sections have multiple negative boundary components and the only positive boundary component is the one associated to $\gamma$ with $mult(\gamma)=1$ and $link(\gamma)>0$. 
\end{example}

We are now ready to prove the main result of this chapter. Let $\phi^t$ be a geodesic flow with the orientation that makes it a positive skewed $\mathbb{R}$-covered Anosov flow.

\begin{theorem}

\label{rcovered}
Any (positive) skewed $\mathbb{R}$-covered Anosov flow obtained by surgery on a geodesic flow is orbit equivalent to a (positive) contact Anosov flow.
\end{theorem}

\begin{proof}
As a consequence of \fullref{FHn} it is enough to show that a $(1,q)$-Goodman surgery along a non-simple closed geodesic does not produce a (positive) skewed $\mathbb{R}$-covered Anosov flow if $q<0$. Given a closed orbit $\gamma$ associated to a non-simple geodesic we chose a partial section as described in \fullref{ex1}. Since $mult(\gamma)=1$ and $link(\gamma)>0$, by \fullref{t2} a $(1,q)$-Goodman surgery along $\gamma$ with $q<0$ yields a negative partial section. By \fullref{t1} the new flow cannot be (positive) skewed $\mathbb{R}$-covered.

\end{proof}

%Let $\phi^t$ geodesic flow with the orientation that makes it a positive skewed $\mathbb{R}$-covered Anosov flow. For a non-simple closed geodesic Marty \cite{Ma} shows that there is a $\overline{p}<0$ such that every $(1,p)$-Goodman surgery with $p<\overline{p}$ generates an Anosov flow that is not positively skewed $\mathbb{R}$-covered. 

\begin{remark}
One important ingredient in proving \fullref{rcovered} is the fact that a $(1,q)$-Goodman surgery along a closed non simple geodesic $\gamma$ in a (positive) skewed $\mathbb{R}$-cover Anosov flow does not produce a (positive) skewed $\mathbb{R}$-covered Anosov flows if $q<0$. We now state a counterpart of this statement in the contact category. 
\begin{proposition}
It is not possible to construct a contact Anosov flow performing Foulon-Hasselblatt construction if the closed geodesic is non simple and $q<0$.
\end{proposition}
\begin{proof}[Sketch of the proof]

The idea is the following. To a neighborhood $N$ of $L$ \fullref{thm:3} associates a positive number $k$ with the property that a $(1,q)$-Goodman surgery produces a contact Anosov flow if $k+q>0$. A different neighborhood $N'$ is associated to a different $k'$. For instance a sequence of nested neighborhoods $  N'\subset N'' \subset N''' $ is associate to a sequence of increasing  slopes $k'''>k''>k'$. For a fixed $K$ we denote with $\overline{k}$ the largest of all the possible values of $k$. If the geodesic is closed and simple there is a infinite sequence of nested neighborhoods as in remark 6.3 and 
$\overline{k}=\infty.$
If the geodesic is non simple we claim
that $\overline{k}\leq 1$. A neighborhood $N$ of $L$ with largest slope $\overline{k}$ contains the closed orbit $\gamma$ in its boundary $\partial N$. As a curve in $\partial N$ it has the form $r\mu+s\lambda$ where $\lambda$ is homotopic to $L$. Since also $\gamma$ is homotopic to $L$ we have $s=1$. Since the geodesic is non simple, the associated closed orbit intersects $C$ and the slope of the characteristic foliation is finite at least in a neighborhood of the point of self intersection. Therefore $\overline{k}\neq\infty$ implying that $\overline{k}=\frac{1}{r}$. Finally, since $r\in \mathbb{Z}$ we have $\overline{k}\leq 1$. This shows that it does not exist a neighborhood $N$ of $L$ such that $k+q>0$ if  $q<0$.
\end{proof}
\end{remark}

%\begin{figure}
%\label{bundles2}
%\labellist 
%     \begin{footnotesize}
 %    \pinlabel $L$ at 55 174
 %    \pinlabel $C_{out}$ at 550 100
%    \pinlabel $A_0$ at 50 144
 %     \pinlabel $w$ at 520 150
 %     \pinlabel $A_{\epsilon}$ at 520 194
 %     \pinlabel $A_{-\epsilon}$ at 525 40
 %    \end{footnotesize}
%\endlabellist
%\includegraphics[width=0.9\textwidth]{figures/Geo3}
%  \caption{   }
%\label{bundles2}
%`\end{figure}

%%%%%%%%%%%%%%%%%%%%%%%%%%%%%%%%%%%%%%%%%%%%%%%%%%%%%%%
\bibliographystyle{alpha}      % amsalpha & amsplain
\bibliography{Bibliography}
%%%%%%%%%%%%%%%%%%%%%%%%%%%%%%%%%%%%%%%%%%%%%%%%%%%%%%%

\end{document}